\documentclass[11pt,letterpaper,twoside,reqno]{amsart}

\usepackage{amsmath}
\usepackage{amssymb}
\usepackage{bm}
\usepackage{mathrsfs}
\usepackage[dvips]{graphicx}

\usepackage{color}

\usepackage{url}

\usepackage{supertabular}

\usepackage{alg}

\makeatletter
\def\algspacing{\alg@unmargin}
\makeatother

\newtheorem{bigthm}{Theorem}

\newtheorem{thm}{Theorem}
\newtheorem{cor}[thm]{Corollary}
\newtheorem{lemma}[thm]{Lemma}
\newtheorem{prop}[thm]{Proposition}

\newtheorem{rem}[thm]{Remark}

\numberwithin{equation}{section}

\numberwithin{thm}{section}

\DeclareMathAlphabet{\mathsfsl}{OT1}{cmss}{m}{sl}

\newcommand{\term}{\emph}

\newcommand{\cnst}[1]{\mathrm{#1}}

	{\makebox{\phantom{}} \\ %
		\textsc{Input:}\begin{itemize}}%
	{\end{itemize}}

	{\textsc{Output:}\begin{itemize}}%
	{\end{itemize}}

	{\textsc{Procedure:}\begin{enumerate}}%
	{\end{enumerate}}

\renewcommand{\phi}{\varphi}

\newcommand{\eps}{\varepsilon}

\newcommand{\defby}{\overset{\mathrm{\scriptscriptstyle{def}}}{=}}

\newcommand{\econst}{\mathrm{e}}

\newcommand{\Id}{\mathbf{I}}

\newcommand{\coll}[1]{\mathscr{#1}}

\newcommand{\Cspace}[1]{\mathbb{C}^{#1}}

\newcommand{\abs}[1]{\left\vert {#1} \right\vert}
\newcommand{\abssq}[1]{{\abs{#1}}^2}

\newcommand{\Prob}[1]{\mathbb{P}\left\{ {#1} \right\}}

\newcommand{\vct}[1]{\bm{#1}}
\newcommand{\mtx}[1]{\bm{#1}}

\newcommand{\adj}{*}
\newcommand{\psinv}{\dagger}

\newcommand{\supp}[1]{\operatorname{supp}(#1)}

\newcommand{\ip}[2]{\left\langle {#1},\ {#2} \right\rangle}

\newcommand{\norm}[1]{\left\Vert {#1} \right\Vert}

\newcommand{\smnorm}[2]{{\bigl\Vert {#2} \bigr\Vert}_{#1}}

\newcommand{\enorm}[1]{\norm{#1}_2}
\newcommand{\enormsq}[1]{\enorm{#1}^2}

\newcommand{\pnorm}[2]{\norm{#2}_{#1}}
\newcommand{\infnorm}[1]{\norm{#1}_{\infty}}

\newcommand{\conv}{\operatorname{conv}}

\newcommand{\subjto}{\quad\text{subject to}\quad}

\newcommand{\bigO}{\mathrm{O}}

\newcommand{\Fee}{\mtx{\Phi}}

\newcommand{\restrict}[1]{\vert_{#1}}

\usepackage{subfigure}
\usepackage{wasysym}

\DeclareMathOperator*{\polylog}{polylog}

\theoremstyle{remark}

\newcommand{\cosamp}{{\rm\sf CoSaMP}}

\begin{document}
\bibliographystyle{plain}
\title[CoSaMP: Iterative Signal Recovery from Noisy Samples]
{CoSaMP: Iterative Signal Recovery \\ from Incomplete and Inaccurate Samples}

\author{D.~Needell and J.~A.~Tropp}

\thanks{DN is with the Mathematics Dept., Univ.~California at Davis, 1 Shields Ave., Davis, CA 95616.
E-mail: \url{dneedell@math.ucdavis.edu}.
JAT is with Applied and Computational Mathematics, MC 217-50, California Inst.~Technology, Pasadena, CA 91125.
E-mail: \url{jtropp@acm.caltech.edu}.}

\date{16 March 2008.  Revised 16 April 2008.  Initially presented at Information Theory and Applications, 31 January 2008, San Diego.}

\begin{abstract}
Compressive sampling offers a new paradigm for acquiring
signals that are compressible with respect to an orthonormal basis. The major
algorithmic challenge in compressive sampling is to approximate a compressible signal from
noisy samples.  
This paper describes a new iterative recovery algorithm called \cosamp\ that
delivers the same guarantees as the best optimization-based approaches.
Moreover, this algorithm offers rigorous bounds on computational cost
and storage.  It is likely to be extremely efficient for practical problems
because it requires only matrix--vector multiplies with the sampling matrix.
For compressible signals, the running time is just $\bigO(N \log^2 N)$,
where $N$ is the length of the signal.
\end{abstract}

\keywords{Algorithms, approximation, basis pursuit, compressed sensing, orthogonal matching pursuit, restricted isometry property, signal recovery, sparse approximation, uncertainty principle}

\subjclass[2000]{41A46, 68Q25, 68W20, 90C27.}

\maketitle

\section{Introduction}

Most signals of interest contain scant information relative to their ambient dimension, but the classical approach to signal acquisition ignores this fact.  We usually collect a complete representation of the target signal and process this representation to sieve out the actionable information.  Then we discard the rest.  Contemplating this ugly inefficiency, one might ask if it is possible instead to acquire \term{compressive samples}.  In other words, is there some type of measurement that automatically winnows out the information from a signal?  Incredibly, the answer is sometimes {\em yes}.

\term{Compressive sampling} refers to the idea that, for certain types of signals, a small number of nonadaptive samples carries sufficient information to approximate the signal well.  Research in this area has two major components:

\begin{description} \setlength{\itemsep}{0.25pc}
\item	[Sampling]
	How many samples are necessary to reconstruct signals to a specified precision?  What type of samples?  How can these sampling schemes be implemented in practice?

\item	[Reconstruction]
	Given the compressive samples, what algorithms can efficiently construct a signal approximation?
\end{description}

\noindent
The literature already contains a well-developed theory of sampling, which we summarize below.  Although algorithmic work has been progressing, the state of knowledge is less than complete.  We assert that a practical signal reconstruction algorithm should have all of the following properties.

\begin{itemize}
\item	It should accept samples from a variety of sampling schemes.

\item	It should succeed using a minimal number of samples.

\item	It should be robust when samples are contaminated with noise.

\item	It should provide optimal error guarantees for every target signal.

\item	It should offer provably efficient resource usage.
\end{itemize}

\noindent
To our knowledge, no approach in the literature simultaneously accomplishes all five goals.

This paper presents and analyzes a novel signal reconstruction algorithm that achieves these desiderata.  The algorithm is called \cosamp, from the acrostic \term{Compressive Sampling Matching Pursuit}.  As the name suggests, the new method is ultimately based on orthogonal matching pursuit (OMP) \cite{TG07:Signal-Recovery}, but it incorporates several other ideas from the literature to accelerate the algorithm and to provide strong guarantees that OMP cannot.  Before we describe the algorithm, let us deliver an introduction to the theory of compressive sampling.


\subsection{Rudiments of Compressive Sampling}

To enhance intuition, we focus on sparse and compressible signals.  For vectors in $\Cspace{N}$, define the $\ell_0$ quasi-norm
$$
\pnorm{0}{ \vct{x} } = \abs{ \supp{ \vct{x} }}= \abs{ \{ j : x_j \neq 0 \} }.
$$
We say that a signal $\vct{x}$ is \term{$s$-sparse} when $\pnorm{0}{\vct{x}} \leq s$.  Sparse signals are an idealization that we do not encounter in applications, but real signals are quite often \term{compressible}, which means that their entries decay rapidly when sorted by magnitude.  As a result, compressible signals are well approximated by sparse signals.  
We can also talk about signals that are compressible with respect to other orthonormal bases, such as a Fourier or wavelet basis.  In this case, the sequence of coefficients in the orthogonal expansion decays quickly.  It represents no loss of generality to focus on signals that are compressible with respect to the standard basis, and we do so without regret.
For a more precise definition of compressibility, turn to Section~\ref{sec:unrecoverable}.

In the theory of compressive sampling, a \term{sample} is a linear functional applied to a signal.  The process of collecting multiple samples is best viewed as the action of a \term{sampling matrix} $\Fee$ on the target signal.  If we take $m$ samples, or \term{measurements}, of a signal in $\Cspace{N}$, then the sampling matrix $\Fee$ has dimensions $m \times N$.  A natural question now arises:
{\em How many measurements are necessary to acquire $s$-sparse signals?}

The minimum number of measurements $m \geq 2s$ on account of the following simple argument.  The sampling matrix must not map two different $s$-sparse signals to the same set of samples.  Therefore, each collection of $2s$ columns from the sampling matrix must be nonsingular.  It is easy to see that certain Vandermonde matrices satisfy this property, 
but these matrices are not really suitable for signal acquisition because they contain square minors that are very badly conditioned.  As a result, some sparse signals are mapped to very similar sets of samples, and it is unstable to invert the sampling process numerically.

Instead, Cand{\`e}s and Tao proposed the stronger condition that the geometry of sparse signals should be preserved under the action of the sampling matrix \cite{CT06:Near-Optimal}.  To quantify this idea, they defined the $r$th \term{restricted isometry constant} of a matrix $\Fee$ as the least number $\delta_r$ for which
\begin{equation} \label{eqn:rip}
(1 - \delta_r) \enormsq{ \vct{x} }
	\leq \enormsq{ \Fee \vct{x} }
	\leq (1 + \delta_r) \enormsq{ \vct{x} }
\qquad\text{whenever $\pnorm{0}{\vct{x}} \leq r$.}
\end{equation}
We have written $\enorm{\cdot}$ for the $\ell_2$ vector norm.  When $\delta_r < 1$, these inequalities imply that each collection of $r$ columns from $\Fee$ is nonsingular, which is the minimum requirement for acquiring $(r/2)$-sparse signals.  When $\delta_r \ll 1$, the sampling operator very nearly maintains the $\ell_2$ distance between each pair of $(r/2)$-sparse signals.  In consequence, it is possible to invert the sampling process stably.

To acquire $s$-sparse signals, one therefore hopes to achieve a small restricted isometry constant $\delta_{2s}$ with as few samples as possible.  A striking fact is that many types of random matrices have excellent restricted isometry behavior.  For example, we can often obtain $\delta_{2s} \leq 0.1$ with
$$
m = \bigO( s \log^{\alpha} N )
$$
measurements, where $\alpha$ is a small integer.  Unfortunately, no deterministic sampling matrix is known to satisfy a comparable bound.  Even worse, it is computationally difficult to check the inequalities \eqref{eqn:rip}, so it may never be possible to exhibit an explicit example of a good sampling matrix.

As a result, it is important to understand how random sampling matrices behave.  The two quintessential examples are Gaussian matrices and partial Fourier matrices.
\begin{description} \setlength{\itemsep}{0.5pc}
\item[Gaussian matrices]
	If the entries of $\sqrt{m} \Fee$ are independent and identically distributed standard normal variables then
	\begin{equation} \label{eqn:gauss-rip}
	m \geq \frac{\cnst{C} r \log(N/r)}{\eps^2}
	\qquad\Longrightarrow\qquad
	\delta_r \leq \eps
	\end{equation}
	except with probability $\econst^{-\cnst{c}m}$.  See \cite{CT06:Near-Optimal} for details.
	

\item[Partial Fourier matrices]
	If $\sqrt{m} \Fee$ is a uniformly random set of $m$ rows drawn from the $N \times N$ unitary discrete Fourier transform (DFT), then
	\begin{equation} \label{eqn:dft-rip}
	m \geq \frac{\cnst{C} r \log^5 N \cdot \log( \eps^{-1} )}{ \eps^2 }
	\qquad\Longrightarrow\qquad
	\delta_r \leq \eps
	\end{equation}
	except with probability $N^{-1}$.  See~\cite{RV06:Sparse-Reconstruction} for the proof.
	Experts believe that the power on the first logarithm should be no greater than two~\cite{Paj08:Personal-Communication}.
\end{description}
Here and elsewhere, we follow the analyst's convention that upright letters ($\cnst{c}, \cnst{C}, \dots$) refer to positive, universal constants that may change from appearance to appearance.

The Gaussian matrix is important because it has optimal restricted isometry behavior.  Indeed, for any $m \times N$ matrix,
$$
\delta_r \leq 0.1
\qquad\Longrightarrow\qquad
m \geq \cnst{C} r \log(N/r),
$$
on account of profound geometric results of Kashin \cite{Kas77:The-widths} and Garnaev--Gluskin \cite{GG84:On-widths}.
Even though partial Fourier matrices may require additional samples to achieve a small restricted isometry constant, they are more interesting for the following reasons \cite{CRT06:Robust-Uncertainty}.
\begin{itemize}
\item	There are technologies that acquire random Fourier measurements at unit cost per sample.

\item	The sampling matrix can be applied to a vector in time $\bigO(N \log N)$.

\item	The sampling matrix requires only $\bigO(m \log N)$ storage.
\end{itemize}
Other types of sampling matrices, such as the \term{random demodulator} \cite{Tro07:Beyond-Nyquist-Talk}, enjoy similar qualities.  These traits are essential for the translation of compressive sampling from theory into practice.

\subsection{Signal Recovery Algorithms}


The major algorithmic challenge in compressive sampling is to approximate a signal given a vector of noisy samples.  The literature describes a huge number of approaches to solving this problem.  They fall into three rough categories:

\begin{description} \setlength{\itemsep}{0.3pc}
\item	[Greedy pursuits]
	These methods build up an approximation one step at a time by making locally optimal choices at each step.  Examples include OMP~\cite{TG07:Signal-Recovery}, stagewise OMP (StOMP)~\cite{DTDS06:Sparse-Solution}, and regularized OMP (ROMP)~\cite{NV07:Uniform-Uncertainty,NV07:ROMP-Stable}.
	
\item	[Convex relaxation]
	These techniques solve a convex program whose minimizer is known to approximate the target signal.  Many algorithms have been proposed to complete the optimization, including interior-point methods~\cite{CRT06:Robust-Uncertainty,KKL+06:Method-Large-Scale}, projected gradient methods~\cite{FNW07:Gradient-Projection}, and iterative thresholding~\cite{DDM04:Iterative-Thresholding}.  

\item	[Combinatorial algorithms]
	These methods acquire highly structured samples of the signal that support rapid reconstruction via group testing.  This class includes Fourier sampling~\cite{GGIMS02:Near-Optimal-Sparse,GMS05:Improved}, chaining pursuit~\cite{GSTV07:Algorithmic}, and HHS pursuit~\cite{GSTV07:HHS}, as well as some algorithms of Cormode--Muthukrishnan~\cite{CM05:Combinatorial} and Iwen~\cite{I07:sub-linear}.
\end{description}

At present, each type of algorithm has its native shortcomings.  Many of the combinatorial algorithms are extremely fast---sublinear in the length of the target signal---but they require a large number of somewhat unusual samples that may not be easy to acquire.  At the other extreme, convex relaxation algorithms succeed with a very small number of measurements, but they tend to be computationally burdensome.  Greedy pursuits---in particular, the ROMP algorithm---are intermediate in their running time and sampling efficiency.

\cosamp, the algorithm described in this paper, is at heart a greedy pursuit.  It also incorporates ideas from the combinatorial algorithms to guarantee speed and to provide rigorous error bounds \cite{GSTV07:HHS}.  The analysis is inspired by the work on ROMP~\cite{NV07:Uniform-Uncertainty,NV07:ROMP-Stable} and the work of Cand{\`e}s--Romberg--Tao \cite{CRT06:Stable} on convex relaxation methods.  In particular, we establish the following result.

\begin{bigthm}[\cosamp] \label{thm:cosamp}
Suppose that $\Fee$ is an $m \times N$ sampling matrix with restricted isometry constant $\delta_{2s} \leq \cnst{c}$.  Let $\vct{u} = \Fee \vct{x} + \vct{e}$ be a vector of samples of an arbitrary signal, contaminated with arbitrary noise.  For a given precision parameter $\eta$, the algorithm \cosamp\ produces a $2s$-sparse approximation $\vct{a}$ that satisfies
$$
\enorm{ \vct{x} - \vct{a} } \leq
	\cnst{C} \cdot \max\left\{ \eta, 
	\frac{1}{\sqrt{s}} \pnorm{1}{\vct{x} - \vct{x}_s} + \enorm{ \vct{e} }
	\right\}
$$
where $\vct{x}_s$ is a best $s$-sparse approximation to $\vct{x}$.
The running time is $\bigO( \coll{L} \cdot \log ( \enorm{\vct{x}} / \eta ) )$, where $\coll{L}$ bounds the cost of a matrix--vector multiply with $\Fee$ or $\Fee^\adj$.  The working storage use is $\bigO(N)$.
\end{bigthm}

Let us expand on the statement of this result.  First, recall that many types of random sampling matrices satisfy the restricted isometry hypothesis when the number of samples $m = \bigO( s \log^{\alpha} N )$.  Therefore, the theorem applies to a wide class of sampling schemes when the number of samples is proportional to the target sparsity and logarithmic in the ambient dimension of the signal space.


The algorithm produces a $2s$-sparse approximation whose $\ell_2$ error is comparable with the scaled $\ell_1$ error of the best $s$-sparse approximation to the signal.  Of course, the algorithm cannot resolve the uncertainty due to the additive noise, so we also pay for the energy in the noise.  This type of error bound is structurally optimal, as we discuss in Section~\ref{sec:unrecoverable}.  Some disparity in the sparsity levels (here, $2s$ versus $s$) seems to be necessary when the recovery algorithm is computationally efficient~\cite{RG08:Sampling-Bounds}.



We can interpret the error guarantee as follows.  In the absence of noise, the algorithm can recover an $s$-sparse signal to arbitrarily high precision.  Performance degrades gracefully as the energy in the noise increases.  Performance also degrades gracefully for compressible signals.  The theorem is ultimately vacuous for signals that cannot be approximated by sparse signals, but compressive sampling is not an appropriate technique for this class.

The running time bound indicates that each matrix--vector multiplication reduces the error by a constant factor (if we amortize over the entire execution).  That is, the algorithm has linear convergence%
\footnote{Mathematicians sometimes refer to linear convergence as ``exponential convergence.''}.
We find that the total runtime is roughly proportional to (the negation of) the \term{reconstruction signal-to-noise ratio} 
$$
\textsf{R-SNR} = 10 \log_{10} \left( \frac{ \enorm{ \vct{x} - \vct{a} } }{ \enorm{\vct{x}} } \right)
\quad\text{dB}.
$$
For compressible signals, one can show that $\abs{\textsf{R-SNR}} = \bigO(\log s)$.  The runtime is also proportional to the cost of a matrix--vector multiply.  For sampling matrices with a fast multiply, the algorithm is accelerated substantially.  In particular, for the partial Fourier matrix, a matrix--vector multiply requires time $\bigO(N \log N)$.  It follows that the total runtime is $\bigO( N \log N \cdot \abs{\textsf{R-SNR}})$.  For most signals of interest,
this cost is nearly linear in the signal length!

\subsection{Notation}

Let us instate several pieces of notation that are carried throughout the paper.  For $p \in [1, \infty]$, we write $\pnorm{p}{\cdot}$ for the usual $\ell_p$ vector norm.  We reserve the symbol $\norm{ \cdot }$ for the spectral norm, i.e., the natural norm on linear maps from $\ell_2$ to $\ell_2$.

Suppose that $\vct{x}$ is a signal in $\Cspace{N}$ and $r$ is a positive integer.  We write $\vct{x}_r$ for the signal in $\Cspace{N}$ that is formed by restricting $\vct{x}$ to its $r$ largest-magnitude components.  Ties are broken lexicographically.  This signal is a best $r$-sparse approximation to $\vct{x}$ with respect to any $\ell_p$ norm.  Suppose now that $T$ is a subset of $\{1, 2, \dots, N\}$.  We define the restriction of the signal to the set $T$ as
$$
\vct{x}\restrict{T} = \begin{cases}
x_i, & i \in T \\
0, & \text{otherwise}.
\end{cases}
$$
We occasionally abuse the notation and treat $\vct{x}\restrict{T}$ as an element of the vector space $\Cspace{T}$.  We also define the restriction $\Fee_T$ of the sampling matrix $\Fee$ as the column submatrix whose columns are listed in the set $T$.

Finally, we define the pseudoinverse of a tall, full-rank matrix $\mtx{A}$ by the formula $\mtx{A}^\psinv = (\mtx{A}^\adj \mtx{A})^{-1} \mtx{A}^\adj$.

\subsection{Organization}

The rest of the paper has the following structure. In Section~\ref{sec:cosamp} we introduce the \cosamp\ algorithm, we state the major theorems in more detail, and we discuss implementation and resource requirements.  Section~\ref{sec:rip} describes some consequences of the restricted isometry property that pervade our analysis.  The central theorem is established for sparse signals in Sections~\ref{sec:invar-sparse} and~\ref{sec:iterative-ls}.  We extend this result to general signals in Section~\ref{sec:cosamp-pf}.  Finally, Section~\ref{sec:compare} places the algorithm in the context of previous work.  The first appendix presents variations on the algorithm.  The second appendix contains a bound on the number of iterations required when the algorithm is implemented using exact arithmetic.

\section{The \cosamp\ Algorithm}\label{sec:cosamp}

This section gives an overview of the algorithm, along with explicit pseudocode.  It presents the major theorems on the performance of the algorithm.  Then it covers details of implementation and bounds on resource requirements.

\subsection{Intuition}

The most difficult part of signal reconstruction is to identify the locations of the largest components in the target signal.  \cosamp\ uses an approach inspired by the restricted isometry property.  Suppose that the sampling matrix $\Fee$ has restricted isometry constant $\delta_s \ll 1$.  For an $s$-sparse signal $\vct{x}$, the vector $\vct{y} = \Fee^\adj \Fee \vct{x}$ can serve as a proxy for the signal because the energy in each set of $s$ components of $\vct{y}$ approximates the energy in the corresponding $s$ components of $\vct{x}$.  In particular, the largest $s$ entries of the proxy $\vct{y}$ point toward the largest $s$ entries of the signal $\vct{x}$.  Since the samples have the form $\vct{u} = \Fee \vct{x}$, we can obtain the proxy just by applying the matrix $\Fee^\adj$ to the samples.

The algorithm invokes this idea iteratively to approximate the target signal.  At each iteration, the current approximation induces a residual, the part of the target signal that has not been approximated.  As the algorithm progresses, the samples are updated so that they reflect the current residual.  These samples are used to construct a proxy for the residual, which permits us to identify the large components in the residual.  This step yields a tentative support for the next approximation.  We use the samples to estimate the approximation on this support set using least squares.  This process is repeated until we have found the recoverable energy in the signal.

\subsection{Overview}
As input, the \cosamp\ algorithm requires four pieces of information:
\begin{itemize}
\item	Access to the sampling operator via matrix--vector multiplication.

\item	A vector of (noisy) samples of the unknown signal.

\item	The sparsity of the approximation to be produced.

\item	A halting criterion.
\end{itemize}

\noindent
The algorithm is initialized with a trivial signal approximation, which means that the initial residual equals the unknown target signal.  During each iteration, \cosamp\ performs five major steps:


\begin{enumerate} \setlength{\itemsep}{0.5pc}
\item	{\bf Identification.}  The algorithm forms a proxy of the residual from the current samples and locates the largest components of the proxy.


\item	{\bf Support Merger.}  The set of newly identified components is united with the set of components that appear in the current approximation.

\item	{\bf Estimation.}  The algorithm solves a least-squares problem to approximate the target signal on the merged set of components.


\item	{\bf Pruning.}  The algorithm produces a new approximation by retaining only the largest entries in this least-squares signal approximation.


\item	{\bf Sample Update.}  Finally, the samples are updated so that they reflect the residual, the part of the signal that has not been approximated.
\end{enumerate}

\noindent
These steps are repeated until the halting criterion is triggered.  In the body of this work, we concentrate on methods that use a fixed number of iterations.  Appendix~\ref{app:variations} discusses some other simple stopping rules that may also be useful in practice.

Pseudocode for \cosamp\ appears as Algorithm~\ref{alg:cosamp}.  This code describes the version of the algorithm that we analyze in this paper.  Nevertheless, there are several adjustable parameters that may improve performance: the number of components selected in the identification step and the number of components retained in the pruning step.  For a brief discussion of other variations on the algorithm, turn to Appendix~\ref{app:variations}.




\vfill

\begin{algorithm}[thb]
\caption{\cosamp\ Recovery Algorithm}
	\label{alg:cosamp}
\centering \fbox{
\begin{minipage}{.9\textwidth} 
\vspace{4pt}
\algname{\cosamp}{$\Fee$, $\vct{u}$, $s$}
\alginout{Sampling matrix $\Fee$, noisy sample vector $\vct{u}$, sparsity level $s$}
{An $s$-sparse approximation $\vct{a}$ of the target signal
}
\vspace{8pt}\hrule\vspace{8pt}

\begin{algtab*}
$\vct{a}^0 \leftarrow \vct{0}$
	 	\hfill \{ Trivial initial approximation \} \\
$\vct{v} \leftarrow \vct{u}$
		 \hfill \{ Current samples = input samples \} \\
$k \leftarrow 0$ \\

\algrepeat
	$k \leftarrow k + 1$ \\

	$\vct{y} \leftarrow \Fee^\adj \vct{v}$
		\hfill \{ Form signal proxy \} \\
	$\Omega \leftarrow \supp{ \vct{y}_{2s} }$
		\hfill \{ Identify large components \} \\	
	$T \leftarrow \Omega \cup \supp{ \vct{a}^{k-1} }$
		\hfill \{ Merge supports \} \\
	
	$\vct{b}\restrict{T} \leftarrow \Fee_T^\psinv \vct{u}$
		\hfill \{ Signal estimation by least-squares \} \\		$\vct{b}\restrict{T^c} \leftarrow \vct{0}$ \\
	
	$\vct{a}^{k} \leftarrow \vct{b}_s$
		\hfill \{ Prune to obtain next approximation \} \\
	$\vct{v} \leftarrow \vct{u} - \Fee \vct{a}^{k}$
		\hfill \{ Update current samples \} \\
		
\alguntil{halting criterion {\it true}}
	

\end{algtab*}
\vspace{-8pt}
\end{minipage}}
\end{algorithm}

\vfill







\clearpage

\subsection{Performance Guarantees}

This section describes our theoretical analysis of the behavior of \cosamp.  The next section covers the resource requirements of the algorithm.  Afterward, Section~\ref{sec:thm-A} combines these materials to establish Theorem~\ref{thm:cosamp}.

Our results depend on a set of hypotheses that has become common in the compressive sampling literature.  Let us frame the standing assumptions:

\vspace{1pc}
\fbox{
\begin{minipage}{0.9\textwidth}
\centering
\textsc{\cosamp\ Hypotheses} \\
\begin{itemize}
\item	The sparsity level $s$ is fixed.
\item	The $m \times N$ sampling operator $\Fee$ has restricted isometry constant $\delta_{4s} \leq 0.1$. \\
\item	The signal $\vct{x} \in \Cspace{N}$ is arbitrary, except where noted.
\item	The noise vector $\vct{e} \in \Cspace{m}$ is arbitrary.
\item	The vector of samples $\vct{u} = \Fee \vct{x} + \vct{e}$.
\end{itemize}
\end{minipage}} 
\vspace{1pc}

\noindent
We also define the \term{unrecoverable energy} $\nu$ in the signal.  This quantity measures the baseline error in our approximation that occurs because of noise in the samples or because the signal is not sparse.

\vspace{1pc}
\fbox{
\begin{minipage}{0.9\textwidth}
\begin{equation} \label{eqn:unrecoverable}
\nu = \enorm{ \vct{x} - \vct{x}_s }
	+ \frac{1}{\sqrt{s}} \pnorm{1}{ \vct{x} - \vct{x}_s }
	+ \enorm{ \vct{e} }.
\end{equation}
\end{minipage}}
\vspace{1pc}

\noindent
We postpone a more detailed discussion of the unrecoverable energy until Section~\ref{sec:unrecoverable}.



Our key result is that \cosamp\ makes significant progress during each iteration where the approximation error is large relative to the unrecoverable energy.  

\begin{thm}[Iteration Invariant] \label{thm:cosamp-invar}
For each iteration $k \geq 0$, the signal approximation $\vct{a}^k$ is $s$-sparse and
$$
\smnorm{2}{ \vct{x} - \vct{a}^{k+1} }
	\leq 0.5 \smnorm{2}{ \vct{x} - \vct{a}^{k} } + 10 \nu.
$$
In particular,
$$
\smnorm{2}{ \vct{x} - \vct{a}^{k} } \leq 2^{-k} \enorm{ \vct{x} } +
	20 \nu.
$$
\end{thm}

\noindent
The proof of Theorem~\ref{thm:cosamp-invar} will occupy us for most of this paper.  In Section~\ref{sec:invar-sparse}, we establish an analog for sparse signals.  The version for general signals appears as a corollary in Section~\ref{sec:cosamp-pf}.




Theorem~\ref{thm:cosamp-invar} has some immediate consequences for the quality of reconstruction with respect to standard signal metrics.  In this setting, a sensible definition of the \term{signal-to-noise ratio} (SNR) is
$$
\textsf{SNR} = 10 \log_{10}
	\left( \frac{ \enorm{ \vct{x} } }{ \nu } \right).
$$
The \term{reconstruction SNR} is defined as
$$
\textsf{R-SNR} = 10 \log_{10} \left(
	\frac{ \enorm{ \vct{x} - \vct{a} } }{ \enorm{ \vct{x} } } \right).
$$
Both quantities are measured in decibels.  Theorem~\ref{thm:cosamp-invar} implies that, after $k$ iterations, the reconstruction SNR satisfies 
$$
\textsf{R-SNR} \apprle 3 - \min\{ 3k, \textsf{SNR} - 13 \}. 
$$
In words, each iteration reduces the reconstruction SNR by about 3 decibels until the error nears the noise floor.  To reduce the error to its minimal value, the number of iterations is proportional to the SNR.

Let us consider a slightly different scenario.  Suppose that the signal $\vct{x}$ is $s$-sparse, so the unrecoverable energy $\nu = \enorm{\vct{e}}$.  Define the \term{dynamic range}
$$
\Delta = 10 \log_{10} \left( 
	\frac{ \max \abs{x_i} }{ \min \abs{x_i} } \right)
\qquad\text{where $i$ ranges over $\supp{\vct{x}}$.}
$$
Assume moreover that the minimum nonzero component of the signal is at least $40\nu$.  Using the fact that $\enorm{\vct{x}} \leq \sqrt{s} \infnorm{\vct{x}}$, it is easy to check that $\enorm{\vct{x} -\vct{a}} \leq \min \abs{x_i}$ as soon as the number $k$ of iterations satisfies
$$
k \apprge 3.3 \Delta + \log_2 \sqrt{s} + 1.
$$
It follows that the support of the approximation $\vct{a}$ must contain every entry in the support of the signal $\vct{x}$.


This discussion suggests that the number of iterations might be substantial if we require a very low reconstruction SNR or if the signal has a very wide dynamic range.  This initial impression is not entirely accurate.  We have established that, when the algorithm performs arithmetic to high enough precision, then a fixed number of iterations suffices to reduce the approximation error to the same order as the unrecoverable energy.  Here is a result for exact computations.

\begin{thm}[Iteration Count] \label{thm:cosamp-count}
Suppose that \cosamp\ is implemented with exact arithmetic.  After at most $6(s+1)$ iterations, \cosamp\ produces an $s$-sparse approximation $\vct{a}$ that satisfies
$$
\enorm{ \vct{x} - \vct{a} } \leq 20 \nu.
$$
\end{thm}

In fact, even more is true.  The number of iterations depends significantly on the structure of the signal.  The only situation where the algorithm needs $\Omega(s)$ iterations occurs when the entries of the signal decay exponentially.  For signals whose largest entries are comparable, the number of iterations may be as small as $\bigO( \log s )$.  This claim is quantified in Appendix \ref{app:count-sparse}, where we prove Theorem~\ref{thm:cosamp-count}.

If we solve the least-squares problems to high precision, an analogous result holds, but the approximation guarantee contains an extra term that comes from solving the least-squares problems imperfectly.  In practice, it may be more efficient overall to solve the least-squares problems to low precision.  The correct amount of care seems to depend on the relative costs of forming the signal proxy and solving the least-squares problem, which are the two most expensive steps in the algorithm.  We discuss this point in the next section.  Ultimately, the question is best settled with empirical studies.

\begin{rem}
In the hypotheses, a bound on the restricted isometry constant $\delta_{2s}$ also suffices.  Indeed, Corollary~\ref{cor:dumb-rip-bd} of the sequel implies that $\delta_{4s} \leq 0.1$ holds whenever $\delta_{2s} \leq 0.025$.  
\end{rem}

\begin{rem}
\noindent
The expression \eqref{eqn:unrecoverable} for the unrecoverable energy can be simplified using Lemma~7 from~\cite{GSTV07:HHS}, which states that, for every signal $\vct{y} \in \Cspace{N}$ and every positive integer $t$, we have
\begin{equation*} \label{eqn:heads-tails}
\enorm{ \vct{y} - \vct{y}_t } \leq \frac{1}{2\sqrt{t}} \pnorm{1}{\vct{y}}.
\end{equation*}
Choosing $\vct{y} = \vct{x} - \vct{x}_{s/2}$ and $t = s/2$, we reach 
\begin{equation} \label{eqn:unrecov-l1}
\nu 
	\leq \frac{1.71}{\sqrt{s}} \pnorm{1}{ \vct{x} - \vct{x}_{s/2} }
	+ \enorm{\vct{e}}.
\end{equation}
In words, the unrecoverable energy is controlled by the scaled $\ell_1$ norm of the signal tail.  
\end{rem}

\subsection{Implementation and Resource Requirements}
	\label{sec:implementation}

\cosamp\ was designed to be a practical method for signal recovery.  An efficient implementation of the algorithm requires 
some ideas from numerical linear algebra, as well as some basic techniques from the theory of algorithms.  This section discusses the key issues and develops an analysis of the running time for the two most common scenarios.

We focus on the least-squares problem in the estimation step because it is the major obstacle to a fast implementation of the algorithm.  The algorithm guarantees that the matrix $\Fee_T$ never has more than $3s$ columns, so our assumption $\delta_{4s} \leq 0.1$ implies that the matrix $\Fee_T$ is extremely well conditioned.  As a result, we can apply the pseudoinverse $\Fee_T^\psinv = (\Fee_T^\adj \Fee_T)^{-1} \Fee_T^\adj$ very quickly using an iterative method, such as Richardson's iteration~\cite[Sec.~7.2.3]{Bjo96:Numerical-Methods} or conjugate gradient~\cite[Sec.~7.4]{Bjo96:Numerical-Methods}.  These techniques have the additional advantage that they only interact with the matrix $\Fee_T$ through its action on vectors.  It follows that the algorithm performs better when the sampling matrix has a fast matrix--vector multiply.

Section~\ref{sec:iterative-ls} contains an analysis of the performance of iterative least-squares algorithms in the context of \cosamp.  In summary, if we initialize the least-squares method with the current approximation $\vct{a}^{k-1}$, then the cost of solving the least-squares problem is $\bigO( \coll{L} )$, where $\coll{L}$ bounds the cost of a matrix--vector multiply with $\Fee_T$ or $\Fee_T^\adj$.  This implementation ensures that Theorem~\ref{thm:cosamp-invar} holds at each iteration.

We emphasize that direct methods for least squares are likely to be extremely inefficient in this setting.  The first reason is that each least-squares problem may contain substantially different sets of columns from $\Fee$.  As a result, it becomes necessary to perform a completely new {\sf QR} or {\sf SVD} factorization during each iteration at a cost of $\bigO(s^2 m)$.  The second problem is that computing these factorizations typically requires direct access to the columns of the matrix, which is problematic when the matrix is accessed through its action on vectors.  Third, direct methods have storage costs $\bigO(sm)$, which may be deadly for large-scale problems.


The remaining steps of the algorithm are standard.  Let us estimate the operation counts.

\begin{description} \setlength{\itemsep}{0.3pc}
\item	[Proxy]
	Forming the proxy is dominated by the cost of the matrix--vector multiply $\Fee^\adj \vct{v}$.

\item	[Identification]
	We can locate the largest $2s$ entries of a vector in time $\bigO(N)$ using
the approach in \cite[Ch.~9]{CLRS01:Intro-Algorithms}.  In practice, it may be faster to sort the entries of the signal in decreasing order of magnitude at cost $\bigO(N \log N)$ and then select the first $2s$ of them.  The latter procedure can be accomplished with quicksort, mergesort, or heapsort~\cite[Sec.~II]{CLRS01:Intro-Algorithms}.  To implement the algorithm to the letter, the sorting method needs to be stable because we stipulate that ties are broken lexicographically. This point is not important in practice.

\item	[Support Merger]
	We can merge two sets of size $\bigO(s)$ in expected time $\bigO(s)$ using randomized hashing methods~\cite[Ch.~11]{CLRS01:Intro-Algorithms}.  One can also sort both sets first and use the elementary merge procedure~\cite[p.~29]{CLRS01:Intro-Algorithms} for a total cost $\bigO(s \log s)$.

\item	[LS estimation]
	We use Richardson's iteration or conjugate gradient to compute $\Fee_T^\psinv \vct{u}$.  Initializing the least-squares algorithm requires a matrix--vector multiply with $\Fee_T^\adj$.  Each iteration of the least-squares method requires one matrix--vector multiply each with $\Fee_T$ and $\Fee_T^\adj$.  Since $\Fee_T$ is a submatrix of $\Fee$, the matrix--vector multiplies can also be obtained from multiplication with the full matrix.  We prove in Section~\ref{sec:iterative-ls} that a constant number of least-squares iterations suffices for Theorem~\ref{thm:cosamp-invar} to hold.
	

\item	[Pruning]
	This step is similar to identification.  Pruning can be implemented in time $\bigO(s)$, but it may be preferable to sort the components of the vector by magnitude and then select the first $s$ at a cost of $\bigO(s\log s)$.

\item	[Sample Update]
	This step is dominated by the cost of the multiplication of $\Fee$ with the $s$-sparse vector $\vct{a}^k$.
\end{description}

\noindent
Table~\ref{tab:runtime} summarizes this discussion in two particular cases.  The first column shows what happens when the sampling matrix $\Fee$ is applied to vectors in the standard way, but we have random access to submatrices.  The second, column shows what happens when the sampling matrix $\Fee$ and its adjoint $\Fee^\adj$ both have a fast multiply with cost $\coll{L}$, where we assume that $\coll{L} \geq N$.  A typical value is $\coll{L} = \bigO(N \log N)$.  In particular, a partial Fourier matrix satisfies this bound.

\begin{table}[thb]
\centering
\renewcommand{\arraystretch}{1.25}
\caption{Operation count for \cosamp.  Big-O notation is omitted for legibility.  The dimensions of the sampling matrix  $\Fee$ are $m \times N$; the sparsity level is $s$.  The number $\coll{L}$ bounds the cost of a matrix--vector multiply with $\Fee$ or $\Fee^\adj$.}  
	\label{tab:runtime}
\vspace{1pc}
\begin{tabular}{|l||r|r|}
\hline
Step				& Standard multiply		& Fast multiply \\
\hline\hline
Form proxy			& $mN$ 					& $\coll{L}$ \\
Identification		& $N$					& $N$ \\
Support merger		& $s$					& $s$ \\
LS estimation		& $sm$ 					& $\coll{L}$ \\
Pruning				& $s$ 					& $s$ \\
Sample update			& $sm$					& $\coll{L}$ \\
\hline
Total per iteration & $\bigO(mN)$ 			& $\bigO(\coll{L})$ \\
\hline
\end{tabular}
\vspace{1pc}
\end{table}

Finally, we note that the storage requirements of the algorithm are also favorable.  Aside from the storage required by the sampling matrix, the algorithm constructs only one vector of length $N$, the signal proxy.  The sample vectors $\vct{u}$ and $\vct{v}$ have length $m$, so they require $\bigO(m)$ storage.  The signal approximations can be stored using sparse data structures, so they require at most $\bigO(s \log N)$ storage.  Similarly, the index sets that appear require only $\bigO(s \log N)$ storage.  The total storage is $\bigO(N)$.

The following result summarizes this discussion.

\begin{thm}[Resource Requirements] \label{thm:resources}
Each iteration of \cosamp\ requires $\bigO(\coll{L})$ time, where $\coll{L}$ bounds the cost of a multiplication with the matrix $\Fee$ or $\Fee^\adj$.  The algorithm uses storage $\bigO(N)$.
\end{thm}

\begin{rem}
We have been able to show that Theorem~\ref{thm:cosamp-count} holds when the least-squares problems are solved iteratively with a delicately chosen stopping threshold.  In this case, the total number of least-squares iterations performed over the entire execution of the algorithm is at most $\bigO( \log(\enorm{\vct{x}} / \eta ) )$ if we wish to achieve error $\bigO( \eta + \nu )$.  When the cost of forming the signal proxy is much higher than the cost of solving the least-squares problem, this analysis may yield a sharper result.  For example, using standard matrix--vector multiplication, we have a runtime bound
$$
\bigO( s \cdot mN + \log(\enorm{\vct{x}} / \eta) \cdot sm ).
$$
The first term reflects the number of \cosamp\ iterations times the cost of forming the signal proxy.  The second term reflects the total cost of the least-squares iterations.  Unless the relative precision $\enorm{\vct{x}}/\eta$ is superexponential in the signal length, we obtain running time $\bigO( smN )$.  This bound is comparable with the worst-case cost of OMP or ROMP.  As we discuss in Appendix \ref{app:count-sparse}, the number of \cosamp\ iterations may be much smaller than $s$, which also improves the estimate.
\end{rem}

\subsection{Proof of Theorem~\ref{thm:cosamp}} \label{sec:thm-A}

We have now collected all the material we need to establish the main result.
Fix a precision parameter $\eta$.  After at most $\bigO( \log(\enorm{\vct{x}} / \eta ) )$ iterations, \cosamp\ produces an $s$-sparse approximation $\vct{a}$ that satisfies
$$
\enorm{ \vct{x} - \vct{a} } \leq \cnst{C} \cdot (\eta + \nu)
$$
in consequence of Theorem~\ref{thm:cosamp-invar}.  Apply inequality \eqref{eqn:unrecov-l1} to bound the unrecoverable energy $\nu$ in terms of the $\ell_1$ norm.  We see that the approximation error satisfies
$$
\enorm{ \vct{x} - \vct{a} } \leq \cnst{C} \cdot \max\left\{ \eta, \frac{1}{\sqrt{s}} \pnorm{1}{ \vct{x} - \vct{x}_{s/2}} + \enorm{\vct{e}} \right\}.
$$
According to Theorem~\ref{thm:resources}, each iteration of \cosamp\ is completed in time $\bigO(\coll{L})$, where $\coll{L}$ bounds the cost of a matrix--vector multiplication with $\Fee$ or $\Fee^\adj$.  The total runtime, therefore, is $\bigO( \coll{L} \log(\enorm{\vct{x}} /\eta) )$.  The total storage is $\bigO(N)$.

In the statement of the theorem, perform the substitution $s/2 \mapsto s$.  Finally, we replace $\delta_{8s}$ with $\delta_{2s}$ by means of Corollary~\ref{cor:dumb-rip-bd}, which states that $\delta_{cr} \leq c \cdot \delta_{2r}$ for any positive integers $c$ and $r$.


\subsection{The Unrecoverable Energy}
	\label{sec:unrecoverable}


Since the unrecoverable energy $\nu$ plays a central role in our analysis of \cosamp, it merits some additional discussion.
In particular, it is informative to examine the unrecoverable energy in a compressible signal.  Let $p$ be a number in the interval $(0,1)$.  We say that $\vct{x}$ is \term{$p$-compressible} with magnitude $R$ if the sorted components of the signal decay at the rate
$$
\abs{x}_{(i)} \leq R \cdot i^{-1/p}
\qquad\text{for $i = 1, 2, 3, \dots$}.
$$
When $p = 1$, this definition implies that $\pnorm{1}{\vct{x}} \leq R \cdot (1 + \log N)$.  Therefore, the unit ball of $1$-compressible signals is similar to the $\ell_1$ unit ball.  When $p \approx 0$, this definition implies that $p$-compressible signals are very nearly sparse.  In general, compressible signals are well approximated by sparse signals:
\begin{align*}
\pnorm{1}{ \vct{x} - \vct{x}_s } &\leq \cnst{C}_p \cdot R \cdot s^{1 - 1/p} \\
\enorm{ \vct{x} - \vct{x}_s } &\leq \cnst{D}_p \cdot R \cdot s^{1/2 - 1/p}
\end{align*}
where $\cnst{C}_p = (1/p - 1)^{-1}$ and $\cnst{D}_p = (2/p - 1)^{-1/2}$.  These results follow by writing each norm as a sum and approximating the sum with an integral.  We see that the unrecoverable energy \eqref{eqn:unrecoverable} in a $p$-compressible signal is bounded as
\begin{equation} \label{eqn:compressible-error}
\nu \leq 2 \cnst{C}_p \cdot R \cdot s^{1/2 - 1/p} + \enorm{\vct{e}}.
\end{equation}
When $p$ is small, the first term in the unrecoverable energy decays rapidly as the sparsity level $s$ increases.  For the class of $p$-compressible signals, the bound \eqref{eqn:compressible-error} on the unrecoverable energy is sharp, modulo the exact values of the constants.

With these inequalities, we can see that \cosamp\ recovers compressible signals efficiently.  Let us calculate the number of iterations required to reduce the approximation error from $\enorm{\vct{x}}$ to the optimal level \eqref{eqn:compressible-error}.  For compressible signals, the energy $\enorm{\vct{x}} \leq 2R$, so
$$
\log\left( \frac{2R}{2\cnst{C}_p \cdot R \cdot s^{1/2 - 1/p}} \right)
	= \log (1/p - 1) + (1/p - 1/2) \log s.
$$
Therefore, the number of iterations required to recover a generic $p$-compressible signal is $\bigO(\log s)$, where the constant in the big-O notation depends on $p$.

The term ``unrecoverable energy'' is justified by several facts.  First, we must pay for the $\ell_2$ error contaminating the samples.  To check this point, define $S = \supp{\vct{x}_s}$.  The matrix $\Fee_S$ is nearly an isometry from $\ell_2^S$ to $\ell_2^m$, so an error in the large components of the signal induces an error of equivalent size in the samples.  Clearly, we can never resolve this uncertainty.

The term $s^{-1/2} \pnorm{1}{\vct{x} - \vct{x}_{s}}$ is also required on account of classical results about the Gel'fand widths of the $\ell_1^N$ ball in $\ell_2^N$, due to Kashin~\cite{Kas77:The-widths} and Garnaev--Gluskin~\cite{GG84:On-widths}.  In the language of compressive sampling, their work has the following interpretation.  Let $\Fee$ be a fixed $m \times N$ sampling matrix.  Suppose that, for every signal $\vct{x} \in \Cspace{N}$, there is an algorithm that uses the samples $\vct{u} = \Fee \vct{x}$ to construct an approximation $\vct{a}$ that achieves the error bound
$$
\enorm{ \vct{x} - \vct{a} }
	\leq \frac{\cnst{C}}{\sqrt{s}} \pnorm{1}{\vct{x}}.
$$
Then the number $m$ of measurements must satisfy
$m \geq \cnst{c} s \log(N/s)$.


\section{Restricted Isometry Consequences}\label{sec:rip}

When the sampling matrix satisfies the restricted isometry inequalities \eqref{eqn:rip}, it has several other properties that we require repeatedly in the proof that the \cosamp\ algorithm is correct.
Our first observation is a simple translation of \eqref{eqn:rip} into other terms.

\begin{prop} \label{prop:rip-basic}
Suppose $\Fee$ has restricted isometry constant $\delta_r$.  Let $T$ be a set of $r$ indices or fewer.  Then 
\begin{align*}
\phantom{\frac{1}{\sqrt{\delta_r}}}
\smnorm{2}{ \Fee_T^\adj \vct{u} }
	&\leq \sqrt{1 + \delta_r} \enorm{ \vct{u} } \\
\phantom{\frac{1}{\sqrt{\delta_r}}}
\smnorm{2}{ \Fee_T^\psinv \vct{u} }
	&\leq \frac{1}{\sqrt{1 - \delta_r}} \enorm{ \vct{u} } \\
\phantom{\frac{1}{\sqrt{\delta_r}}}
\smnorm{2}{ \Fee_T^\adj \Fee_T \vct{x} }
	&\lesseqqgtr (1 \pm \delta_r) \enorm{ \vct{x}} \\
\phantom{\frac{1}{\sqrt{\delta_r}}}
\smnorm{2}{ (\Fee_T^\adj \Fee_T)^{-1} \vct{x} }
	&\lesseqqgtr \frac{1}{1 \pm \delta_r} \enorm{ \vct{x}}.
\end{align*}
where the last two statements contain an upper and lower bound, depending on the sign chosen.
\end{prop}


\begin{proof}
The restricted isometry inequalities \eqref{eqn:rip} imply that the singular values of $\Fee_T$ lie between $\sqrt{1 - \delta_r}$ and $\sqrt{1 + \delta_r}$. The bounds follow from standard relationships between the singular values of $\Fee_T$ and the singular values of basic functions of $\Fee_T$.
%
\end{proof}

A second consequence is that disjoint sets of columns from the sampling matrix span nearly orthogonal subspaces.  The following result quantifies this observation.

\begin{prop}[Approximate Orthogonality] \label{prop:approx-orth}
Suppose $\Fee$ has restricted isometry constant $\delta_r$.  Let $S$ and $T$ be disjoint sets of indices whose combined cardinality does not exceed $r$.  Then
$$
\norm{ \Fee_S^\adj \Fee_T } \leq \delta_{r}.
$$
\end{prop}

\begin{proof}
Abbreviate $R = S \cup T$, and observe that $\Fee_S^\adj \Fee_T$ is a submatrix of $\Fee_R^\adj \Fee_R - \Id$.  The spectral norm of a submatrix never exceeds the norm of the entire matrix.  We discern that
$$
\norm{ \Fee_S^\adj \Fee_T }
	\leq \norm{ \Fee_R^\adj \Fee_R - \Id }
	\leq \max\{ (1+\delta_r) - 1, 1 - (1 - \delta_r) \}
	= \delta_r
$$
because the eigenvalues of $\Fee_R^\adj \Fee_R$ lie between $1 - \delta_r$ and $1 + \delta_r$.
\end{proof}


This result will be applied through the following corollary.

\begin{cor} \label{cor:cross-corr}
Suppose $\Fee$ has restricted isometry constant $\delta_r$.  Let $T$ be a set of indices, and let $\vct{x}$ be a vector.  Provided that $r \geq \abs{T \cup \supp{\vct{x}}}$,
$$
\enorm{ \Fee_T^\adj \Fee \cdot \vct{x}\restrict{T^c} }
	\leq \delta_{r} \enorm{ \vct{x}\restrict{T^c} }.
$$
\end{cor}

\begin{proof}
Define $S = \supp{\vct{x}} \setminus T$, so we have $\vct{x}\restrict{S} = \vct{x}\restrict{T^c}$.  Thus,
$$
\enorm{ \Fee_T^\adj \Fee \cdot \vct{x}\restrict{T^c} }
	= \enorm{ \Fee_T^\adj \Fee \cdot \vct{x}\restrict{S} }
	\leq \norm{ \Fee_T^\adj \Fee_S } \enorm{ \vct{x}\restrict{S} }
	\leq \delta_{r} \enorm{ \vct{x}\restrict{T^c} },
$$
owing to Proposition~\ref{prop:approx-orth}.
\end{proof}

As a second corollary, we show that $\delta_{2r}$ gives weak control over the higher restricted isometry constants.

\begin{cor} \label{cor:dumb-rip-bd}
Let $c$ and $r$ be positive integers.  Then $\delta_{cr} \leq c \cdot \delta_{2r}$.
\end{cor}

\begin{proof}
The result is clearly true for $c = 1, 2,$ so we assume $c \geq 3$.  Let $S$ be an arbitrary index set of size $cr$, and let $\mtx{M} = \Fee_S^\adj \Fee_S - \Id$.  It suffices to check that $\norm{ \mtx{M} } \leq c \cdot \delta_{2r}$.  To that end, we break the matrix $\mtx{M}$ into $r \times r$ blocks, which we denote $\mtx{M}_{ij}$.  A block version of Gershgorin's theorem states that $\norm{\mtx{M}}$ satisfies at least one of the inequalities
$$
\abs{ \norm{ \mtx{M} } - \norm{\mtx{M}_{ii}} } \leq \sum\nolimits_{j\neq i} \norm{ \mtx{M}_{ij} }
\qquad\text{where $i = 1, 2, \dots, c$.}
$$
The derivation is entirely analogous with the usual proof of Gershgorin's theorem, so we omit the details.  For each diagonal block, we have $\norm{ \mtx{M}_{ii} } \leq \delta_r$ because of the restricted isometry inequalities \eqref{eqn:rip}.  For each off-diagonal block, we have $\norm{ \mtx{M}_{ij} } \leq \delta_{2r}$ because of Proposition~\ref{prop:approx-orth}.  Substitute these bounds into the block Gershgorin theorem and rearrange to complete the proof.
\end{proof}

Finally, we present a result that measures how much the sampling matrix inflates nonsparse vectors.  This bound permits us to establish the major results for sparse signals and then transfer the conclusions to the general case.

\begin{prop}[Energy Bound] \label{prop:k2-bd}
Suppose that $\Fee$ verifies the upper inequality of \eqref{eqn:rip}, viz.
$$
\enorm{ \Fee \vct{x} } \leq \sqrt{1 + \delta_r} \enorm{ \vct{x} }
\qquad\text{when}\qquad
\pnorm{0}{\vct{x}} \leq r.
$$
Then, for every signal $\vct{x}$,
$$
\enorm{ \Fee \vct{x} } \leq \sqrt{1 + \delta_r}
	\left[ \enorm{ \vct{x} } + \frac{1}{\sqrt{r}}
		\pnorm{1}{\vct{x}} \right].
$$
\end{prop}

\begin{proof}
We repeat the geometric argument of Rudelson that is presented in~\cite{GSTV07:HHS}.

First, observe that the hypothesis of the proposition can be regarded as a statement about the operator norm of $\Fee$ as a map between two Banach spaces.  For a set $I \subset \{1, 2, \dots, N\}$, write $B_2^I$ for the unit ball in $\ell_2(I)$.  Define the convex body
$$
S = \conv\left\{ \bigcup\nolimits_{\abs{I} \leq r} B_2^I \right\}
	\subset \Cspace{N},
$$
and notice that, by hypothesis, the operator norm
$$
\pnorm{S \to 2}{ \Fee } =
	\max_{\vct{x} \in S} \enorm{ \Fee \vct{x} }
	\leq \sqrt{1 + \delta_r}.
$$
Define a second convex body 
$$
K = \left\{ \vct{x} : \enorm{\vct{x}} + \frac{1}{\sqrt{r}} \pnorm{1}{\vct{x}}
	\leq 1 \right\} \subset \Cspace{N},
$$
and consider the operator norm
$$
\pnorm{K \to 2}{ \Fee } =
	\max_{\vct{x} \in K} \enorm{ \Fee \vct{x} }.
$$
The content of the proposition is the claim that
$$
\pnorm{K \to 2}{ \Fee } \leq \pnorm{S \to 2}{ \Fee }.
$$
To establish this point, it suffices to check that $K \subset S$.

Instead, we prove the reverse inclusion for the polars: $S^\circ \subset K^\circ$.  The norm with unit ball $S^\circ$ is calculated as
$$
\pnorm{S^\circ}{ \vct{u} }
	= \max_{\abs{I} \leq r} \enorm{ \vct{u}\restrict{I} }.
$$
Consider a vector $\vct{u}$ in the unit ball $S^\circ$, and let $I$ be a set of $r$ coordinates where $\vct{u}$ is largest in magnitude.  We must have
$$
\pnorm{\infty}{\vct{u}\restrict{I^c}} \leq \frac{1}{\sqrt{r}},
$$
or else $\abs{u_i} > \frac{1}{\sqrt{r}}$ for each $i \in I$.  But then $\pnorm{S^\circ}{ \vct{u} } \geq \enorm{ \vct{u}\restrict{I} } > 1$, a contradiction.  Therefore, we may write
$$
\vct{u} = \vct{u}\restrict{I} + \vct{u}\restrict{I^c}
	\in B_2 + \frac{1}{\sqrt{r}} B_\infty,
$$
where $B_p$ is the unit ball in $\ell_p^N$.  But the set on the right-hand side is precisely the unit ball of $K^\circ$ since
\begin{align*}
\sup\nolimits_{\vct{v} \in K^\circ} \ip{ \vct{x} }{ \vct{v} }
	&= \pnorm{K}{ \vct{x} }
	= \enorm{ \vct{x} } + \frac{1}{\sqrt{r}} \pnorm{1}{\vct{x}} \\
	&= \sup\nolimits_{\vct{v} \in B_2} \ip{ \vct{x} }{ \vct{v} }
		+ \sup\nolimits_{\vct{w} \in \frac{1}{\sqrt{r}} B_\infty} \ip{ \vct{x} }{ \vct{w} }
	= \sup\nolimits_{\vct{v} \in B_2 + \frac{1}{\sqrt{r}} B_\infty} \ip{ \vct{x} }{ \vct{v} }.
\end{align*}
In summary, $S^\circ \subset K^\circ$.
\end{proof}





\section{The Iteration Invariant: Sparse Case}\label{sec:invar-sparse}

We now commence the proof of Theorem~\ref{thm:cosamp-invar}.  For the moment, let us assume that the signal is actually sparse.  Section~\ref{sec:cosamp-pf} removes this assumption.



The result states that each iteration of the algorithm reduces the approximation error by a constant factor, while adding a small multiple of the noise.  As a consequence, when the approximation error is large in comparison with the noise, the algorithm makes substantial progress in identifying the unknown signal.

\begin{thm}[Iteration Invariant: Sparse Case]
	\label{thm:invar-sparse}
Assume that $\vct{x}$ is $s$-sparse.  For each $k \geq 0$, the signal approximation $\vct{a}^k$ is $s$-sparse, and
$$
\smnorm{2}{ \vct{x} - \vct{a}^{k+1} }
	\leq 0.5 \smnorm{2}{ \vct{x} - \vct{a}^{k} } + 7.5 \enorm{ \vct{e} }.
$$
In particular,
$$
\smnorm{2}{ \vct{x} - \vct{a}^{k} }
	\leq 2^{-k} \enorm{ \vct{x} } + 15 \enorm{ \vct{e} }.
$$
\end{thm}

\noindent
The argument proceeds in a sequence of short lemmas, each corresponding to one step in the algorithm.  Throughout this section, we retain the assumption that $\vct{x}$ is $s$-sparse.





\subsection{Approximations, Residuals, etc.}


Fix an iteration $k \geq 1$.  We write $\vct{a} = \vct{a}^{k-1}$ for the signal approximation at the beginning of the iteration.  Define the residual $\vct{r} = \vct{x} - \vct{a}$, which we interpret as the part of the signal we have not yet recovered.  Since the approximation $\vct{a}$ is always $s$-sparse, the residual $\vct{r}$ must be $2s$-sparse.  Notice that the vector $\vct{v}$ of updated samples can be viewed as noisy samples of the residual:
$$
\vct{v} \defby \vct{u} - \Fee \vct{a} = \Fee (\vct{x} - \vct{a}) + \vct{e}
	= \Fee \vct{r} + \vct{e}.
$$

\subsection{Identification}

The identification phase produces a set of components where the residual signal still has a lot of energy.  

\begin{lemma}[Identification] \label{lem:ident}
The set $\Omega = \supp{ \vct{y}_{2s} }$ contains at most $2s$ indices, and
$$
\enorm{ \vct{r}\restrict{\Omega^c} }
	\leq 0.2223 \enorm{ \vct{r} } + 2.34 \enorm{\vct{e}}.
$$
\end{lemma}

\begin{proof}
The identification phase forms a proxy $\vct{y} = \Fee^\adj \vct{v}$ for the residual signal.  The algorithm then selects a set $\Omega$ of $2s$ components from $\vct{y}$ that have the largest magnitudes.  The goal of the proof is to show that the energy in the residual on the set $\Omega^c$ is small in comparison with the total energy in the residual.

Define the set $R = \supp{\vct{r}}$.  Since $R$ contains at most $2s$ elements, our choice of $\Omega$ ensures that $\enorm{ \vct{y}\restrict{R} } \leq \enorm{ \vct{y}\restrict{\Omega} }$.  By squaring this inequality and canceling the terms in $R \cap \Omega$, we discover that
$$
\enorm{ \vct{y}\restrict{R \setminus \Omega} }
	\leq \enorm{ \vct{y}\restrict{\Omega \setminus R} }.
$$
Since the coordinate subsets here contain few elements, we can use the restricted isometry constants to provide bounds on both sides.

First, observe that the set $\Omega \setminus R$ contains at most $2s$ elements.  Therefore, we may apply Proposition~\ref{prop:rip-basic} and Corollary~\ref{cor:cross-corr} to obtain
\begin{align*}
\enorm{ \vct{y}\restrict{\Omega \setminus R} }
	&= \smnorm{2}{ \Fee_{\Omega \setminus R}^\adj (\Fee \vct{r} + \vct{e}) } \\
	&\leq \smnorm{2}{ \Fee_{\Omega \setminus R}^\adj \Fee \vct{r} }
		+ \smnorm{2}{ \Fee_{\Omega \setminus R}^\adj \vct{e} } \\
	&\leq \delta_{4s} \enorm{ \vct{r} }
		+ \sqrt{1 + \delta_{2s}} \enorm{ \vct{e} }.
\end{align*}
Likewise, the set $R \setminus \Omega$ contains $2s$ elements or fewer, so Proposition~\ref{prop:rip-basic} and Corollary~\ref{cor:cross-corr} yield
\begin{align*}
\enorm{ \vct{y}\restrict{R \setminus \Omega} }
	&= \smnorm{2}{ \Fee_{R \setminus \Omega}^\adj (\Fee \vct{r} + \vct{e}) } \\
	&\geq \smnorm{2}{ \Fee_{R \setminus \Omega}^\adj \Fee
			\cdot \vct{r}\restrict{R \setminus \Omega} }
		- \smnorm{2}{ \Fee_{R \setminus \Omega}^\adj \Fee
			\cdot \vct{r}\restrict{\Omega} }
		- \smnorm{2}{ \Fee_{R \setminus \Omega}^\adj \vct{e} } \\
	&\geq (1 - \delta_{2s}) \smnorm{2}{ \vct{r}\restrict{R \setminus \Omega} }
		- \delta_{2s} \enorm{ \vct{r} }
		- \sqrt{1 + \delta_{2s}} \enorm{ \vct{e} }. 
\end{align*}
Since the residual is supported on $R$, we can rewrite $\vct{r}\restrict{R \setminus \Omega} = \vct{r}\restrict{\Omega^c}$.  Finally, combine the last three inequalities and rearrange to obtain
$$
\enorm{ \vct{r}\restrict{\Omega^c} }
	\leq \frac{ (\delta_{2s} + \delta_{4s}) \enorm{ \vct{r} }
		+ 2 \sqrt{1 + \delta_{2s}} \enorm{ \vct{e} } }
		{ 1 - \delta_{2s} }.
$$
Invoke the numerical hypothesis that $\delta_{2s} \leq \delta_{4s} \leq 0.1$ to complete the argument.
\end{proof}

\subsection{Support Merger}

The next step of the algorithm merges the support of the current signal approximation $\vct{a}$ with the newly identified set of components.  The following result shows that components of the signal $\vct{x}$ outside this set have very little energy.

\begin{lemma}[Support Merger] \label{lem:merger}
Let $\Omega$ be a set of at most $2s$ indices.  The set $T = \Omega \cup \supp{\vct{a}}$ contains at most $3s$ indices, and
$$
\enorm{ \vct{x}\restrict{T^c} }
	\leq \enorm{ \vct{r}\restrict{\Omega^c} }.
$$
\end{lemma}



\begin{proof}
Since $\supp{\vct{a}} \subset T$, we find that
$$
\enorm{ \vct{x}\restrict{T^c} }
	= \enorm{ (\vct{x} - \vct{a})\restrict{T^c} }
	= \enorm{ \vct{r}\restrict{T^c} }
	\leq \enorm{ \vct{r}\restrict{\Omega^c} },
$$
where the inequality follows from the containment $T^c \subset \Omega^c$.
\end{proof}

\subsection{Estimation}

The estimation step of the algorithm solves a least-squares problem to obtain values for the coefficients in the set $T$.  We need a bound on the error of this approximation.

\begin{lemma}[Estimation] \label{lem:estimation}
Let $T$ be a set of at most $3s$ indices, and define the least-squares signal estimate $\vct{b}$ by the formulae
$$
\vct{b}\restrict{T} = \Fee_T^\psinv \vct{u}
\qquad\text{and}\qquad
\vct{b}\restrict{T^c} = \vct{0}.
$$
Then
$$
\enorm{ \vct{x} - \vct{b} }
	\leq 1.112 \enorm{ \vct{x}\restrict{T^c} } + 1.06 \enorm{\vct{e}}.
$$
\end{lemma}

This result assumes that we solve the least-squares problem in infinite precision.  In practice, the right-hand side of the bound contains an extra term owing to the error from the iterative least-squares solver.  In Section~\ref{sec:iterative-ls}, we study how many iterations of the least-squares solver are required to make the least-squares error negligible in the present argument.

\begin{proof}
Note first that
$$
\enorm{ \vct{x} - \vct{b} }
	\leq \enorm{ \vct{x}\restrict{T^c} }
		+ \enorm{ \vct{x}\restrict{T} - \vct{b}\restrict{T} }.
$$
Using the expression $\vct{u} = \Fee\vct{x} + \vct{e}$ and the fact $\Fee_T^\psinv \Fee_T = \Id_T$, we calculate that
\begin{align*}
\enorm{ \vct{x}\restrict{T} - \vct{b}\restrict{T} }
	&= \smnorm{2}{ \vct{x}\restrict{T} - \Fee_T^\psinv(\Fee \cdot \vct{x}\restrict{T} + \Fee \cdot \vct{x}\restrict{T^c} + \vct{e}) } \\
	&=  \smnorm{2}{ \Fee_T^\psinv (\Fee \cdot \vct{x}\restrict{T^c}
		+ \vct{e}) } \\
	&\leq \smnorm{2}{ (\Fee_T^\adj \Fee_T)^{-1} \Fee_T^\adj \Fee \cdot \vct{x}\restrict{T^c} } + \smnorm{2}{ \Fee_T^\psinv \vct{e} }.
\end{align*}
The cardinality of $T$ is at most $3s$, and $\vct{x}$ is $s$-sparse, so Proposition~\ref{prop:rip-basic} and Corollary~\ref{cor:cross-corr} imply that
\begin{align*}
\enorm{ \vct{x}\restrict{T} - \vct{b}\restrict{T} }
	&\leq \frac{1}{1 - \delta_{3s}} \enorm{ \Fee_T^\adj \Fee \cdot \vct{x}\restrict{T^c} } + \frac{1}{\sqrt{1 - \delta_{3s}}} \enorm{ \vct{e}} \\
	&\leq \frac{\delta_{4s}}{1 - \delta_{3s}}
		\enorm{ \vct{x}\restrict{T^c} }
		+ \frac{\enorm{\vct{e}}}{\sqrt{1 - \delta_{3s}}}.
\end{align*}
Combine the bounds to reach
$$
\enorm{ \vct{x} - \vct{b} } \leq
	\left[ 1 + \frac{\delta_{4s}}{1 - \delta_{3s}} \right]
		\enorm{ \vct{x}\restrict{T^c} }
	+ \frac{\enorm{\vct{e}}}{\sqrt{1 - \delta_{3s}}}.
$$
Finally, invoke the hypothesis that $\delta_{3s} \leq \delta_{4s} \leq 0.1$.
\end{proof}

\subsection{Pruning}

The final step of each iteration is to prune the intermediate approximation to its largest $s$ terms.  The following lemma provides a bound on the error in the pruned approximation.

\begin{lemma}[Pruning] \label{lem:pruning}
The pruned approximation $\vct{b}_s$ satisfies
$$
\enorm{ \vct{x} - \vct{b}_s }
	\leq 2 \enorm{ \vct{x} - \vct{b} }.
$$
\end{lemma}

\begin{proof}
The intuition is that $\vct{b}_s$ is close to $\vct{b}$, which is close to $\vct{x}$.  Rigorously,
$$
\enorm{ \vct{x} - \vct{b}_s }
	\leq \enorm{ \vct{x} - \vct{b} } + \enorm{ \vct{b} - \vct{b}_s }
	\leq 2 \enorm{ \vct{x} - \vct{b} }.
$$
The second inequality holds because $\vct{b}_s$ is the best $s$-sparse approximation to $\vct{b}$.  In particular, the $s$-sparse vector $\vct{x}$ is a worse approximation.
\end{proof}

\subsection{Proof of Theorem~\ref{thm:invar-sparse}}
	\label{sec:pf-invar-sparse}

We now complete the proof of the iteration invariant for sparse signals, Theorem~\ref{thm:invar-sparse}.  At the end of an iteration, the algorithm forms a new approximation $\vct{a}^{k} = \vct{b}_s$, which is evidently $s$-sparse.  Applying the lemmas we have established, we easily bound the error:
\begin{align*}
\smnorm{2}{ \vct{x} - \vct{a}^k }
	&= \enorm{ \vct{x} - \vct{b}_s } \\
	&\leq 2 \enorm{ \vct{x} - \vct{b} }
		&& \text{Pruning (Lemma \ref{lem:pruning})} \\
	&\leq 2 \cdot ( 1.112 \enorm{ \vct{x}\restrict{T^c} }
		+ 1.06\enorm{\vct{e}} )
		&& \text{Estimation (Lemma \ref{lem:estimation})} \\
	&\leq 2.224 \enorm{ \vct{r}\restrict{\Omega^c} } + 2.12 \enorm{\vct{e}}
		&& \text{Support merger (Lemma \ref{lem:merger})} \\
	&\leq 2.224 \cdot ( 0.2223 \enorm{ \vct{r} } + 2.34 \enorm{\vct{e}} )
		+ 2.12 \enorm{\vct{e}}
		&& \text{Identification (Lemma \ref{lem:ident})} \\
	&< 0.5 \enorm{ \vct{r} } + 7.5 \enorm{\vct{e}} \\
	&= 0.5 \smnorm{2}{ \vct{x} - \vct{a}^{k-1} } + 7.5 \enorm{\vct{e}}.
\end{align*}
To obtain the second bound in Theorem~\ref{thm:invar-sparse}, simply solve the error recursion and note that
$$
(1 + 0.5 + 0.25 + \dots) \cdot 7.5 \enorm{\vct{e}} \leq 15 \enorm{\vct{e}}.
$$
This point completes the argument.

\section{Analysis of Iterative Least-squares}
	\label{sec:iterative-ls}

To develop an efficient implementation of \cosamp, it is critical to use an iterative method when we solve the least-squares problem in the estimation step.  The two natural choices are Richardson's iteration and conjugate gradient.  The efficacy of these methods rests on the assumption that the sampling operator has small restricted isometry constants.  Indeed, since the set $T$ constructed in the support merger step contains at most $3s$ components, the hypothesis $\delta_{4s} \leq 0.1$ ensures that the condition number
$$
\kappa( \Fee_T^\adj \Fee_T )
	= \frac{\lambda_{\max}(\Fee_T^\adj \Fee_T)}{\lambda_{\min}(\Fee_T^\adj \Fee_T)}
	\leq \frac{1 + \delta_{3s}}{1 - \delta_{3s}}
	< 1.223.
$$
This condition number is closely connected with the performance of Richardson's iteration and conjugate gradient.  In this section, we show that Theorem~\ref{thm:invar-sparse} holds if we perform a constant number of iterations of either least-squares algorithm.

\subsection{Richardson's Iteration}

For completeness, let us explain how Richardson's iteration can be applied to solve the least-squares problems that arise in \cosamp.  Suppose we wish to compute $\mtx{A}^\psinv \vct{u}$ where $\mtx{A}$ is a tall, full-rank matrix.  Recalling the definition of the pseudoinverse, we realize that this amounts to solving a linear system of the form
$$
(\mtx{A}^\adj \mtx{A}) \vct{b} = \mtx{A}^\adj \vct{u}.
$$
This problem can be approached by \term{splitting} the Gram matrix:
$$
\mtx{A}^\adj \mtx{A} = \Id + \mtx{M}
$$
where $\mtx{M} = \mtx{A}^\adj \mtx{A} - \Id$.  Given an initial iterate $\vct{z}^0$, Richardon's method produces  the subsequent iterates via the formula
$$
\vct{z}^{\ell+1} = \mtx{A}^\adj \vct{u} - \mtx{M} \vct{z}^{\ell}.
$$
Evidently, this iteration requires only matrix--vector multiplies with $\mtx{A}$ and $\mtx{A}^\adj$.  It is worth noting that Richardson's method can be accelerated \cite[Sec.~7.2.5]{Bjo96:Numerical-Methods}, but we omit the details.

It is quite easy to analyze Richardson's iteration~\cite[Sec.~7.2.1]{Bjo96:Numerical-Methods}.  Observe that
$$
\smnorm{2}{ \vct{z}^{\ell+1} - \mtx{A}^\psinv \vct{u} }
	= \smnorm{2}{ \mtx{M} ( \vct{z}^\ell - \mtx{A}^{\psinv} \vct{u}) }
	\leq \norm{ \mtx{M} } \smnorm{2}{ \vct{z}^\ell - \mtx{A}^{\psinv} \vct{u} }.
$$
This recursion delivers
$$
\smnorm{2}{ \vct{z}^{\ell} - \mtx{A}^\psinv \vct{u} }
	\leq \norm{ \mtx{M} }^\ell \smnorm{2}{ \vct{z}^0 - \mtx{A}^\psinv \vct{u} }
\qquad\text{for $\ell = 0, 1, 2, \dots$.}
$$
In words, the iteration converges linearly.

In our setting, $\mtx{A} = \Fee_T$ where $T$ is a set of at most $3s$ indices.  Therefore, the restricted isometry inequalities \eqref{eqn:rip} imply that
$$
\norm{ \mtx{M} } = \norm{ \Fee_T^\adj \Fee_T - \Id } \leq \delta_{3s}.
$$
We have assumed that $\delta_{3s} \leq \delta_{4s} \leq 0.1$, which means that the iteration converges quite fast.  Once again, the restricted isometry behavior of the sampling matrix plays an essential role in the performance of the \cosamp\ algorithm.

Conjugate gradient provides even better guarantees for solving the least-squares problem, but it is somewhat more complicated to describe and rather more difficult to analyze.  We refer the reader to \cite[Sec.~7.4]{Bjo96:Numerical-Methods} for more information.  The following lemma summarizes the behavior of both Richardson's iteration and conjugate gradient in our setting.

\begin{lemma}[Error Bound for LS]
	\label{lem:ls-error}
Richardson's iteration produces a sequence $\{ \vct{z}^\ell \}$ of iterates that satisfy
$$
\smnorm{2}{ \vct{z}^\ell - \Fee_T^\psinv \vct{u} }
	\leq 0.1^\ell \cdot \smnorm{2}{ \vct{z}^0 - \Fee_T^\psinv \vct{u} }
\qquad\text{for $\ell = 0, 1, 2, \dots$}. 
$$ 
Conjugate gradient produces a sequence of iterates that satisfy
$$
\smnorm{2}{ \vct{z}^\ell - \Fee_T^\psinv \vct{u} }
	\leq 2 \cdot \rho^\ell \cdot \smnorm{2}{ \vct{z}^0 - \Fee_T^\psinv \vct{u} }
\qquad\text{for $\ell = 0, 1, 2, \dots$}. 
$$
where
$$
\rho = \frac{ \sqrt{\kappa(\Fee_T^\adj \Fee_T)} - 1 }{\sqrt{\kappa(\Fee_T^\adj \Fee_T)} + 1} \leq 0.072.
$$
\end{lemma}

\subsection{Initialization}

Iterative least-squares algorithms must be seeded with an initial iterate, and their performance depends heavily on a wise selection thereof.  \cosamp\ offers a natural choice for the initializer: the current signal approximation.  As the algorithm progresses, the current signal approximation provides an increasingly good starting point for solving the least-squares problem.


\begin{lemma}[Initial Iterate for LS]
	\label{lem:ls-init}
Let $\vct{x}$ be an $s$-sparse signal with noisy samples $\vct{u} = \Fee \vct{x} + \vct{e}$.  Let $\vct{a}^{k-1}$ be the signal approximation at the end of the $(k-1)$th iteration, and let $T$ be the set of components identified by the support merger. Then
$$
\smnorm{2}{ \vct{a}^{k-1} - \Fee_T^\psinv \vct{u}  }
	\leq 2.112 \smnorm{2}{ \vct{x} - \vct{a}^{k-1} } + 1.06 \enorm{ \vct{e} }
$$
\end{lemma}

\begin{proof}
By construction of $T$, the approximation $\vct{a}^{k-1}$ is supported inside $T$, so
$$
\smnorm{2}{ \vct{x}\restrict{T^c} }
	= \smnorm{2}{(\vct{x} - \vct{a}^{k-1})\restrict{T^c} }
	\leq \smnorm{2}{\vct{x} - \vct{a}^{k-1}}.
$$
Using Lemma~\ref{lem:estimation}, we may calculate how far $\vct{a}^{k-1}$ lies from the solution to the least-squares problem.
\begin{align*}
\smnorm{2}{ \vct{a}^{k-1} - \Fee_T^\psinv \vct{u} }
	&\leq \smnorm{2}{ \vct{x} - \vct{a}^{k-1} }
		+ \smnorm{2}{ \vct{x} - \Fee_T^\psinv \vct{u} } \\
	&\leq \smnorm{2}{ \vct{x} - \vct{a}^{k-1} }
		+ 1.112 \smnorm{2}{ \vct{x}\restrict{T^c} }
		+ 1.06 \enorm{ \vct{e} } \\
	&\leq 2.112 \smnorm{2}{ \vct{x} - \vct{a}^{k-1} } + 1.06 \enorm{ \vct{e} }.
\end{align*}
Roughly, the error in the initial iterate is controlled by the current approximation error.
\end{proof}

\subsection{Iteration Count}

We need to determine how many iterations of the least-squares algorithm are required to ensure that the approximation produced is sufficiently good to support the performance of \cosamp.

\begin{cor}[Estimation by Iterative LS]
	\label{cor:ls-est}
Suppose that we initialize the LS algorithm with $\vct{z}^0 = \vct{a}^{k-1}$.  After at most three iterations, both Richardson's iteration and conjugate gradient produce a signal estimate $\vct{b}$ that satisfies
$$
\smnorm{2}{ \vct{x} - \vct{b} } \leq
	1.112 \smnorm{2}{ \vct{x}\restrict{T^c} }
	+ 0.0022 \smnorm{2}{ \vct{x} - \vct{a}^{k-1} } + 1.062 \enorm{ \vct{e} }.
$$
\end{cor}

\begin{proof}
Combine Lemma~\ref{lem:ls-error} and Lemma~\ref{lem:ls-init} to see that three iterations of Richardson's method yield
$$
\smnorm{2}{ \vct{z}^3 - \Fee_T^\psinv \vct{u}  }
	\leq 0.002112 \smnorm{2}{ \vct{x} - \vct{a}^{k-1} } + 0.00106 \enorm{ \vct{e} }.
$$
The bound for conjugate gradient is slightly better.  Let $\vct{b}\restrict{T} = \vct{z}^3$.  According to the estimation result, Lemma~\ref{lem:estimation}, we have
$$
\smnorm{2}{ \vct{x} - \Fee_T^\psinv \vct{u} } \leq
	1.112 \smnorm{2}{ \vct{x}\restrict{T^c} } + 1.06 \enorm{ \vct{e} }.
$$
An application of the triangle inequality completes the argument.
\end{proof}

\subsection{\cosamp\ with Iterative least-squares}

Finally, we need to check that the sparse iteration invariant, Theorem~\ref{thm:invar-sparse} still holds when we use an iterative least-squares algorithm.

\begin{thm}[Sparse Iteration Invariant with Iterative LS] \label{thm:invar-sparse-ls}
Suppose that we use Richardson's iteration or conjugate gradient for the estimation step, initializing the LS algorithm with the current approximation $\vct{a}^{k-1}$ and performing three LS iterations.  Then Theorem~\ref{thm:invar-sparse} still holds.
\end{thm}

\begin{proof}
We repeat the calculation in Section~\ref{sec:pf-invar-sparse} using Corollary~\ref{cor:ls-est} instead of the simple estimation lemma.  To that end, recall that the residual $\vct{r} = \vct{x} - \vct{a}^{k-1}$.  Then
\begin{align*}
\smnorm{2}{ \vct{x} - \vct{a}^k }
	&\leq 2 \enorm{ \vct{x} - \vct{b} } \\
	&\leq 2 \cdot ( 1.112 \enorm{ \vct{x}\restrict{T^c} }
		+ 0.0022 \enorm{ \vct{r} } + 1.062 \enorm{\vct{e}} ) \\
	&\leq 2.224 \enorm{ \vct{r}\restrict{\Omega^c} } + 0.0044 \enorm{\vct{r}} + 2.124 \enorm{\vct{e}} \\
	&\leq 2.224 \cdot ( 0.2223 \enorm{ \vct{r} } + 2.34 \enorm{\vct{e}} ) + 0.0044 \enorm{\vct{r}}
		+ 2.124 \enorm{\vct{e}} \\
	&< 0.5 \enorm{ \vct{r} } + 7.5 \enorm{\vct{e}} \\
	&= 0.5 \smnorm{2}{ \vct{x} - \vct{a}^{k-1} } + 7.5 \enorm{\vct{e}}.
\end{align*}
This bound is precisely what is required for the theorem to hold.
\end{proof}

\section{Extension to General Signals}
	\label{sec:cosamp-pf}

In this section, we finally complete the proof of the main result for \cosamp, Theorem~\ref{thm:cosamp-invar}.  The remaining challenge is to remove the hypothesis that the target signal is sparse, which we framed in Theorems~\ref{thm:invar-sparse} and~\ref{thm:invar-sparse-ls}.  Although this difficulty might seem large, the solution is simple and elegant.  It turns out that we can view the noisy samples of a general signal as samples of a sparse signal contaminated with a different noise vector that implicitly reflects the tail of the original signal.



\begin{lemma}[Reduction to Sparse Case] \label{lem:reduction}
Let $\vct{x}$ be an arbitrary vector in $\Cspace{N}$.  The sample vector $\vct{u} = \Fee \vct{x} + \vct{e}$ can also be expressed as $\vct{u} = \Fee \vct{x}_s + \widetilde{\vct{e}}$ where
$$
\enorm{ \widetilde{\vct{e}} } \leq
1.05 \left[ \enorm{ \vct{x} - \vct{x}_s } + \frac{1}{\sqrt{s}}
	\pnorm{1}{ \vct{x} - \vct{x}_s } \right] + \enorm{ \vct{e} }.
$$
\end{lemma}


\begin{proof}
Decompose $\vct{x} = \vct{x}_s + (\vct{x} - \vct{x}_s)$ to obtain $\vct{u} = \Fee \vct{x}_s + \widetilde{\vct{e}}$ where $\widetilde{\vct{e}} = \Fee (\vct{x} - \vct{x}_s) + \vct{e}$.
To compute the size of the error term, we simply apply the triangle inequality and Proposition~\ref{prop:k2-bd}:
$$
\enorm{ \widetilde{\vct{e}} }
	\leq \sqrt{1 + \delta_s} \left[ \enorm{ \vct{x} - \vct{x}_s } + \frac{1}{\sqrt{s}}
		\pnorm{1}{ \vct{x} - \vct{x}_s } \right]
	+ \enorm{ \vct{e} }.
$$
Finally, invoke the fact that $\delta_s \leq \delta_{4s} \leq 0.1$ to obtain $\sqrt{1 + \delta_s} \leq 1.05$.
\end{proof}

This lemma is just the tool we require to complete the remaining argument.

\begin{proof}[Proof of Theorem~\ref{thm:cosamp-invar}]
Let $\vct{x}$ be a general signal, and use Lemma~\ref{lem:reduction} to write the noisy vector of samples $\vct{u} = \Fee \vct{x}_s + \widetilde{\vct{e}}$.  Apply the sparse iteration invariant, Theorem~\ref{thm:invar-sparse}, or the analog for iterative least-squares, Theorem~\ref{thm:invar-sparse-ls}.  We obtain
$$
\smnorm{2}{ \vct{x}_s - \vct{a}^{k+1} }
	\leq 0.5 \smnorm{2}{ \vct{x}_s - \vct{a}^k }
		+ 7.5 \enorm{ \widetilde{\vct{e}}}.
$$
Invoke the lower and upper triangle inequalities to obtain
$$
\smnorm{2}{ \vct{x} - \vct{a}^{k+1} }
	\leq 0.5 \smnorm{2}{ \vct{x} - \vct{a}^k }
		+ 7.5 \enorm{ \widetilde{\vct{e}} }
		+ 1.5 \enorm{ \vct{x} - \vct{x}_s }.
$$
Finally, recall the estimate for $\enorm{\widetilde{\vct{e}}}$ from Lemma~\ref{lem:reduction}, and simplify to reach
\begin{align*}
\smnorm{2}{ \vct{x} - \vct{a}^{k+1} }
	&\leq 0.5 \smnorm{2}{ \vct{x} - \vct{a}^k }
		+ 9.375 \enorm{ \vct{x} - \vct{x}_s }
		+ \frac{7.875}{\sqrt{s}} \pnorm{1}{ \vct{x} - \vct{x}_s }
		+ 7.5 \enorm{ \vct{e} } \\
	&< 0.5 \smnorm{2}{ \vct{x} - \vct{a}^k } + 10 \nu.
\end{align*}
where $\nu$ is the unrecoverable energy \eqref{eqn:unrecoverable}.
\end{proof}

\section{Discussion and Related Work}\label{sec:compare}

\cosamp\ draws on both algorithmic ideas and analytic techniques that have appeared before. This section describes the other major signal recovery algorithms, and it compares them with \cosamp.  It also attempts to trace the key ideas in the  algorithm back to their sources.



\subsection{Algorithms for Compressive Sampling}

We begin with a short discussion of the major algorithmic approaches to signal recovery from compressive samples.  We focus on provably correct methods, although we acknowledge that some \term{ad hoc} techniques provide excellent empirical results.

The initial discovery works on compressive sampling proposed to perform signal recovery by solving a convex optimization problem~\cite{CRT06:Robust-Uncertainty,Don06:Compressed-Sensing}.  Given a sampling matrix $\Fee$ and a noisy vector of samples $\vct{u} = \Fee \vct{x} + \vct{e}$, consider the mathematical program
\begin{equation} \label{eqn:bp}
\min \pnorm{1}{ \vct{y} }
\quad\subjto\quad
\Fee \vct{y} = \vct{u}.
\end{equation}
In words, we look for a signal reconstruction that is consistent with the samples but has minimal $\ell_1$ norm.  The intuition behind this approach is that minimizing the $\ell_1$ norm promotes sparsity, so allows the approximate recovery of compressible signals.  Cand{\`e}s, Romberg, and Tao established in~\cite{CRT06:Stable} that a minimizer $\vct{a}$ of \eqref{eqn:bp} satisfies
\begin{equation} \label{eqn:bp-err}
\enorm{ \vct{x} - \vct{a} } \leq \cnst{C} \left[ \frac{1}{\sqrt{s}} \pnorm{1}{\vct{x} - \vct{x}_s }
	+ \enorm{\vct{e}} \right]
\end{equation}
provided that the sampling matrix $\Fee$ has restricted isometry constant $\delta_{4s} \leq 0.2$.  In \cite{Can08:Restricted-Isometry}, the hypothesis on the restricted isometry constant is sharpened to $\delta_{2s} \leq \sqrt{2} - 1$.  The error bound for \cosamp\ is equivalent, modulo the exact value of the constants.

The literature describes a huge variety of algorithms for solving the optimization problem \eqref{eqn:bp}.  The most common approaches involve interior-point methods~\cite{CRT06:Robust-Uncertainty,KKL+06:Method-Large-Scale}, projected gradient methods~\cite{FNW07:Gradient-Projection}, or iterative thresholding~\cite{DDM04:Iterative-Thresholding} 
The interior-point methods are guaranteed to solve the problem to a fixed precision in time $\bigO( m^2 N^{1.5} )$, where $m$ is the number of measurements and $N$ is the signal length~\cite{NN94:Interior-Point}.  Note that the constant in the big-O notation depends on some of the problem data.  The other convex relaxation algorithms, while sometimes faster in practice, do not currently offer rigorous guarantees.  \cosamp\ provides rigorous bounds on the runtime that are much better than the available results for interior-point methods.

Tropp and Gilbert proposed the use of a greedy iterative algorithm called \term{orthogonal matching pursuit} (OMP) for signal recovery \cite{TG07:Signal-Recovery}.  The algorithm initializes the current sample vector $\vct{v} = \vct{u}$.  In each iteration, it forms the signal proxy $\vct{y} = \Fee^\adj \vct{v}$ and identifies a component of the proxy with largest magnitude.  It adds the new component to the set $T$ of previously identified components.  Then OMP forms a new signal approximation by solving a least-squares problem: $\vct{a} = \Fee_T^\psinv \vct{u}$.  Finally, it updates the samples $\vct{v} = \vct{u} - \Fee \vct{a}$.  These steps are repeated until a halting criterion is satisfied.

Tropp and Gilbert were able to prove a weak result for the performance of OMP \cite{TG07:Signal-Recovery}.  Suppose that $\vct{x}$ is a fixed, $s$-sparse signal, and let $m = \cnst{C} s \log N$.  Draw an $m \times N$ sampling matrix $\Fee$ whose entries are independent, zero-mean subgaussian%
\footnote{A subgaussian random variable $Z$ satisfies $\Prob{ \abs{Z} > t } \leq \cnst{c} \econst^{-\cnst{c}t^2}$ for all $t > 0$.}
random variables with equal variances.  Given noiseless measurements $\vct{u} = \Fee \vct{x}$, OMP reconstructs $\vct{x}$ after $s$ iterations, except with probability $N^{-1}$.  In this setting, OMP must fail for some sparse signals, so it does not provide the same uniform guarantees as convex relaxation.  It is unknown whether OMP succeeds for compressible signals or whether it succeeds when the samples are contaminated with noise.

Donoho et al.~invented another greedy iterative method called \term{stagewise OMP}, or StOMP \cite{DTDS06:Sparse-Solution}.  This algorithm uses the signal proxy to select multiple components at each step, using a rule inspired by ideas from wireless communications.  The algorithm is faster than OMP because of the selection rule, and it sometimes provides good performance, although parameter tuning can be difficult.  There are no rigorous results available for StOMP.

Very recently, Needell and Vershynin developed and analyzed another greedy approach, called \term{regularized OMP}, or ROMP \cite{NV07:Uniform-Uncertainty, NV07:ROMP-Stable}.  This algorithm is similar to OMP but uses a more sophisticated selection rule.  Among the $s$ largest entries of the signal proxy, it identifies the largest subset whose entries differ in magnitude by at most a factor of two.  The work on ROMP represents an advance because the authors establish under restricted isometry hypotheses that their algorithm can approximately recover any compressible signal from noisy samples.  More precisely, suppose that the sampling matrix $\Fee$ has restricted isometry constant $\delta_{8s} \leq 0.01 / \sqrt{\log s}$.   Given noisy samples $\vct{u} = \Fee\vct{x} + \vct{e}$, ROMP produces a $2s$-sparse signal approximation $\vct{a}$ that satisfies
$$
\enorm{ \vct{x} - \vct{a} } \leq \cnst{C} \sqrt{\log s} \left[ \frac{1}{\sqrt{s}}
	\pnorm{1}{ \vct{x} - \vct{x}_s } + \enorm{ \vct{e} } \right].
$$
This result is comparable with the result for convex relaxation, aside from the extra logarithmic factor in the restricted isometry hypothesis and the error bound.  The results for \cosamp\ show that it does not suffer these parasitic factors, so its performance is essentially optimal.

After we initially presented this work, Dai and Milenkovic developed an algorithm called Subspace Pursuit that is very similar to \cosamp.  They established that their algorithm offers performance guarantees analogous with those for \cosamp.  See~\cite{DM08:Subspace-Pursuit} for details.

Finally, we note that there is a class of sublinear algorithms for signal reconstruction from compressive samples.  A sublinear algorithm uses time and space resources that are asymptotically smaller than the length of the signal.  One of the earliest such techniques is the Fourier sampling algorithm of Gilbert et al.~\cite{GGIMS02:Near-Optimal-Sparse,GMS05:Improved}.  This algorithm uses random (but structured) time samples to recover signals that are compressible with respect to the discrete Fourier basis.  Given $s \polylog(N)$ samples%
\footnote{The term $\polylog$ indicates a function that is dominated by a polynomial in the logarithm of its argument.},
Fourier sampling produces a signal approximation $\vct{a}$ that satisfies
$$
\enorm{ \vct{x} - \vct{a} } \leq \cnst{C} \enorm{ \vct{x} - \vct{x}_s }
$$
except with probability $N^{-1}$.  The result for Fourier sampling holds for each signal (rather than for all).  Later, Gilbert et al.~developed two other sublinear algorithms, chaining pursuit~\cite{GSTV07:Algorithmic} and HHS pursuit~\cite{GSTV07:HHS}, that offer uniform guarantees for all signals.  Chaining pursuit has an error bound
$$
\pnorm{1}{ \vct{x} - \vct{a} } \leq \cnst{C} \log N \pnorm{1}{ \vct{x} - \vct{x}_s }
$$
which is somewhat worse than \eqref{eqn:bp-err}.  HHS pursuit achieves the error bound \eqref{eqn:bp-err}.  These methods all require more measurements than the linear and superlinear algorithms (by logarithmic factors), and these measurements must be highly structured.  As a result, the sublinear algorithms may not be useful in practice.

The sublinear algorithms are all combinatorial in nature.  They use ideas from group testing to identify the support of the signal quickly.  There are several other combinatorial signal recovery methods due to Cormode--Muthukrishnan~\cite{CM05:Combinatorial} and Iwen~\cite{I07:sub-linear}.  These algorithms have drawbacks similar to the sublinear approaches.

\subsection{Relative Performance}

Table \ref{tab:comparison} summarizes the relative behavior of these algorithms in terms of the following criteria.
\begin{description} \setlength{\itemsep}{0.5pc}
\item	[General samples]
	Does the algorithm work for a variety of sampling schemes?  Or does it require structured samples?  The designation ``RIP'' means that a bound on a restricted isometry constant suffices. ``Subgauss.'' means that the algorithm succeeds for the class of subgaussian sampling matrices.
	
\item	[Optimal number of samples]
	Can the algorithm recover $s$ sparse signals from $\bigO(s \log N)$ measurements?  Or are its sampling requirements higher (by logarithmic factors)?

\item	[Uniformity]
	Does the algorithm recover all signals given a fixed sampling matrix?  Or do the results require a sampling matrix to be drawn at random for each signal?
	
\item	[Stability]
	Does the algorithm succeed when (a) the signal is compressible but not sparse and (b) when the samples are contaminated with noise?  In most cases, stable algorithms have error bounds similar to \eqref{eqn:bp-err}.  See the discussion above for details.
	
\item	[Running time]
	What is the worst-case cost of the algorithm to recover a real-valued $s$-sparse signal to a fixed relative precision, given a sampling matrix with no special structure?  The designation LP($N$, $m$) indicates the cost of solving a linear program with $N$ variables and $m$ constraints, which is $\bigO(m^2 N^{1.5})$ for an interior-point method.  Note that most of the algorithms can also take advantage of fast matrix--vector multiplies to obtain better running times.
\end{description}

Of the linear and superlinear algorithms, \cosamp\ achieves the best performance on all these metrics.  Although \cosamp\ is slower than the sublinear algorithms, it makes up for this shortcoming by allowing more general sampling matrices and requiring fewer samples.

\begin{table}
\begin{tabular}{|l||r|r|r|r|r|r|}
\hline
					& \cosamp	& OMP			& ROMP		& Convex opt.		& Fourier Samp.	& HHS Pursuit \\
\hline
General samples 	& RIP		& Subgauss. 	& RIP		& RIP			& no			& no	\\
Opt.~\# samples	& yes		& yes			& no		& yes			& no			& no	\\
Uniformity			& yes		& no			& yes		& yes			& no			& yes	\\
Stability			& yes		& ?				& yes		& yes			& yes		& yes	\\
Running time		& $\bigO(mN)$	& $\bigO(smN)$		& $\bigO(smN)$	& LP($N$, $m$)	& $s\polylog(N)$	& ${\rm poly}(s\log N)$ \\
\hline
\end{tabular}
\vspace{1pc}

\caption{Comparison of several signal recovery algorithms.  The notation $s$ refers to the sparsity level; $m$ refers the number of measurements; $N$ refers to the signal length.  See the text for comments on specific designations.}
	\label{tab:comparison}
\end{table}

%



\subsection{Key Ideas}

We conclude with a historical overview of the ideas that inform the \cosamp\ algorithm and its analysis.

The overall greedy iterative structure of \cosamp\ has a long history.  The idea of approaching sparse approximation problems in this manner dates to the earliest algorithms.  In particular, methods for variable selection in regression, such as forward selection and its relatives, all take this form \cite{Mil02:Subset-Selection}. Temlyakov's survey~\cite{Tem02:Nonlinear-Methods} describes the historical role of greedy algorithms in nonlinear approximation.  Mallat and Zhang introduced greedy algorithms into the signal processing literature and proposed the name \term{matching pursuit} \cite{MZ93:Matching-Pursuits}.  Gilbert, Strauss, and their collaborators showed how to incorporate greedy iterative strategies into fast algorithms for sparse approximation problems, and they established the first rigorous guarantees for greedy methods \cite{GGIKMS02:Fast-Small-Space,GGIMS02:Near-Optimal-Sparse}.  Tropp provided a new theoretical analysis of OMP in his work \cite{Tro04:Greed-Good}.  Subsequently, Tropp and Gilbert proved that OMP was effective for compressive sampling \cite{TG07:Signal-Recovery}.

Unlike the simplest greedy algorithms, \cosamp\ identifies many components during each iteration, which allows the algorithm to run faster for many types of signals.  It is not entirely clear where this idea first appeared.  Several early algorithms of Gilbert et al.~incorporate this approach \cite{GMS03:Approximation-Functions,TGMS03:Improved-Sparse}, and it is an essential feature of the Fourier sampling algorithm \cite{GGIMS02:Near-Optimal-Sparse, GMS05:Improved}.  More recent compressive sampling recovery algorithms also select multiple indices, including chaining pursuit~\cite{GSTV07:Algorithmic}, HHS pursuit~\cite{GSTV07:HHS}, StOMP~\cite{DTDS06:Sparse-Solution}, and~ROMP~\cite{NV07:Uniform-Uncertainty}.

\cosamp\ uses the restricted isometry properties of the sampling matrix to ensure that the identification step is successful.  Cand{\`e}s and Tao isolated the restricted isometry conditions in their work on convex relaxation methods for compressive sampling \cite{CT05:Decoding-Linear}.  The observation that restricted isometries can also be used to ensure the success of greedy methods is relatively new.  This idea plays a role in HHS pursuit \cite{GSTV07:HHS}, but it is expressed more completely in the analysis of ROMP \cite{NV07:Uniform-Uncertainty}.


The pruning step of \cosamp\ is essential to maintain the sparsity of the approximation, which is what permits us to use restricted isometries in the analysis of the algorithm.  It also has significant ramifications for the running time because it impacts the speed of the iterative least-squares algorithms.  This technique originally appeared in HHS pursuit~\cite{GSTV07:HHS}.


The iteration invariant, Theorem \ref{thm:cosamp-invar}, states that if the error is large then \cosamp\ makes substantial progress.  This approach to the overall analysis echoes the analysis of other greedy iterative algorithms, including the Fourier sampling method~\cite{GGIMS02:Near-Optimal-Sparse,GMS05:Improved} and HHS Pursuit~\cite{GSTV07:HHS}.

Finally, mixed-norm error bounds, such as that in Theorem~\ref{thm:cosamp}, have become an important feature of the compressive sampling literature.  This idea appears in the work of Cand{\`e}s--Romberg--Tao on convex relaxation \cite{CRT06:Stable}; it is used in the analysis of HHS pursuit \cite{GSTV07:HHS}; it also plays a role in the theoretical treatment of Cohen--Dahmen--DeVore \cite{CDD06:Remarks}.

\subsection*{Acknowledgment}
We would like to thank Martin Strauss for many inspiring discussions.  He is ultimately responsible for many of the ideas in the algorithm and analysis.  We would also like to thank Roman Vershynin for suggestions that drastically simplified the proofs.

\appendix

\section{Algorithmic Variations}
	\label{app:variations}

This appendix describes other possible halting criteria and their consequences.  It also proposes some other variations on the algorithm.

\subsection{Halting Rules}

There are three natural approaches to halting the algorithm.  The first, which we have discussed in the body of the paper, is to stop after a fixed number of iterations.  Another possibility is to use the norm $\enorm{ \vct{v} }$ of the current samples as evidence about the norm $\enorm{\vct{r}}$ of the residual.  A third possibility is to use the magnitude $\infnorm{\vct{y}}$ of the entries of the proxy to bound the magnitude $\infnorm{\vct{r}}$ of the entries of the residual.

It suffices to discuss halting criteria for sparse signals because Lemma~\ref{lem:reduction} shows that the general case can be viewed in terms of sampling a sparse signal.  Let $\vct{x}$ be an $s$-sparse signal, and let $\vct{a}$ be an $s$-sparse approximation.  The residual $\vct{r} = \vct{x} - \vct{a}$.  We write $\vct{v} = \Fee \vct{r} + \vct{e}$ for the induced noisy samples of the residual and $\vct{y} = \Fee^\adj \vct{v}$ for the signal proxy. 

The discussion proceeds in two steps.  First, we argue that an \term{a priori} halting criterion will result in a guarantee about the quality of the final signal approximation.


\begin{thm}[Halting I]
The halting criterion $\enorm{ \vct{v} } \leq \eps$ ensures that
$$
\enorm{ \vct{x} - \vct{a} } \leq 1.06 \cdot ( \eps + \enorm{ \vct{e}}).
$$
The halting criterion $\pnorm{\infty}{ \vct{y} } \leq \eta / \sqrt{2s}$ ensures that
$$
\pnorm{\infty}{ \vct{x} - \vct{a} } \leq 1.12 \eta + 1.17 \enorm{\vct{e}}.
$$
\end{thm}

\begin{proof}
Since $\vct{r}$ is $2s$-sparse, Proposition~\ref{prop:rip-basic} ensures that
$$
\sqrt{1 - \delta_{2s}} \enorm{ \vct{r} } - \enorm{ \vct{e} }
	\leq \enorm{ \vct{v} }.
$$
If $\enorm{ \vct{v} } \leq \eps$, it is immediate that
$$
\enorm{ \vct{r} } \leq
\frac{ \eps + \enorm{ \vct{e} } }{\sqrt{1 - \delta_{2s}}}.
$$
The definition $\vct{r} = \vct{x} - \vct{a}$ and the numerical bound $\delta_{2s} \leq \delta_{4s} \leq 0.1$ dispatch the first claim.

Let $R = \supp{ \vct{r} }$, and note that $\abs{R} \leq 2s$.  Proposition~\ref{prop:rip-basic} results in
$$
(1 - \delta_{2s}) \enorm{ \vct{r} } - \sqrt{1 + \delta_{2s}} \enorm{ \vct{e} }
	\leq \enorm{ \vct{y}\restrict{R} }.
$$
Since
$$
\enorm{ \vct{y} \restrict{R} }
	\leq \sqrt{2s} \pnorm{ \infty}{ \vct{y}\restrict{R} }
	\leq \sqrt{2s} \pnorm{ \infty}{ \vct{y} },
$$
we find that the requirement $\pnorm{\infty}{\vct{y}} \leq \eta /\sqrt{2s}$ ensures that
$$
\pnorm{\infty}{ \vct{r} } \leq
	\frac{\eta + \sqrt{1 + \delta_{2s}}\enorm{ \vct{e} }}{1 - \delta_{2s}}.
$$
The numerical bound $\delta_{2s} \leq 0.1$ completes the proof.

\end{proof}

Second, we check that each halting criterion is triggered when the residual has the desired property.

\begin{thm}[Halting II]
The halting criterion $\enorm{ \vct{v} } \leq \eps$ is triggered as soon as
$$
\enorm{ \vct{x} - \vct{a} } \leq 0.95 \cdot ( \eps - \enorm{ \vct{e} } ).
$$
The halting criterion $\pnorm{\infty}{\vct{y}} \leq \eta / \sqrt{2s}$ is triggered as soon as
$$
\pnorm{\infty}{ \vct{x} - \vct{a} } \leq \frac{0.45 \eta}{ s } - \frac{0.68 \enorm{\vct{e}}}{\sqrt{s}}.
$$\end{thm}

\begin{proof}
Proposition \ref{prop:rip-basic} shows that
$$
\enorm{ \vct{v} } \leq \sqrt{1 + \delta_{2s}} \enorm{ \vct{r}} + \enorm{ \vct{e} }.
$$
Therefore, the condition
$$
\enorm{\vct{r}} \leq
	\frac{\eps - \enorm{\vct{e}}}{\sqrt{1 + \delta_{2s}}}
$$
ensures that $\enorm{ \vct{v} } \leq \eps$.  Note that $\delta_{2s} \leq 0.1$ to complete the first part of the argument.

Now let $R$ be the singleton containing the index of a largest-magnitude coefficient of $\vct{y}$.
Proposition~\ref{prop:rip-basic} implies that
$$
\pnorm{ \infty}{ \vct{y} } = \enorm{ \vct{y}\restrict{R} } \leq \sqrt{1 + \delta_{1}}\enorm{ \vct{v} }.
$$
By the first part of this theorem, the halting criterion $\pnorm{\infty}{\vct{y}} \leq \eta / \sqrt{2s}$ is triggered as soon as
$$
\enorm{ \vct{x} - \vct{a} } \leq 0.95 \cdot \left( \frac{\eta}{\sqrt{2s}\sqrt{1 + \delta_{1}}} - \enorm{ \vct{e} } \right).
$$
Since $\vct{x} - \vct{a}$ is $2s$-sparse, we have the bound $\enorm{ \vct{x} - \vct{a} } \leq \sqrt{2s}\pnorm{\infty}{ \vct{x-a} }$.  To wrap up, recall that $\delta_{1} \leq \delta_{2s} \leq 0.1$.
\end{proof}

\subsection{Other Variations}

This section briefly describes several natural variations on \cosamp\ that may improve its performance.

\begin{enumerate} \setlength{\itemsep}{0.5pc}
\item	Here is a version of the algorithm that is, perhaps, simpler than Algorithm~\ref{alg:cosamp}.  At each iteration, we approximate the current residual rather than the entire signal.  This approach is similar to HHS Pursuit \cite{GSTV07:HHS}.  The inner loop changes in the following manner.
\vspace{1pc}
\begin{description} \setlength{\itemsep}{0.5pc}
\item	[Identification]
	As before, select $\Omega = \supp{\vct{y}_{2s}}$.

\item	[Estimation]
	Solve a least-squares problem with the \emph{current samples} instead of the original samples to obtain an approximation of the \emph{residual signal}.  Formally, $\vct{b} = \Fee_\Omega^\psinv \vct{v}$.  In this case, one initializes the iterative least-squares algorithm with the zero vector to take advantage of the fact that the residual is becoming small.

\item	[Approximation Merger]
	Add this approximation of the residual to the previous approximation of the signal to obtain a new approximation of the signal: $\vct{c} = \vct{a}^{k-1} + \vct{b}$.

\item	[Pruning]
	Construct the $s$-sparse signal approximation: $\vct{a}^k = \vct{c}_s$.

\item	[Sample Update]
	Update the samples as before: $\vct{v} = \vct{u} - \Fee \vct{a}^k$.
\end{description}
\vspace{1pc}
We can show that this algorithm satisfies a result similar to Theorem~\ref{thm:cosamp-invar} by adapting the current argument.  We were unable to verify an analog of Theorem~\ref{thm:cosamp-count}.  Nevertheless, we believe that this version of the algorithm is also promising.

\item	After the inner loop of the algorithm is complete, we can solve another least-squares problem in an effort to improve the final result.  If $\vct{a}$ is the approximation at the end of the loop, we set $T = \supp{ \vct{a} }$.  Then solve $\vct{b} = \Fee_T^\psinv \vct{u}$ and output the $s$-sparse signal approximation $\vct{b}$.  Note that the output approximation is not guaranteed to be better than $\vct{a}$ because of the noise vector $\vct{e}$, but it should never be much worse.


\item	Another variation is to prune the merged support $T$ down to $s$ entries \emph{before} solving the least-squares problem.  One may use the values of the proxy $\vct{y}$ as surrogates for the unknown values of the new approximation on the set $\Omega$.  Since the least-squares problem is solved at the end of the iteration, the columns of $\Fee$ that are used in the least-squares approximation are orthogonal to the current samples $\vct{v}$.  As a result, the identification step always selects new components in each iteration.  We have not attempted an analysis of this algorithm.

\end{enumerate}

\section{Iteration Count}\label{app:count-sparse}

In this appendix, we obtain an estimate on the number of iterations of the \cosamp\ algorithm necessary to identify the recoverable energy in a sparse signal, assuming exact arithmetic.  Except where stated explicitly, we assume that $\vct{x}$ is $s$-sparse.  It turns out that the number of iterations depends heavily on the signal structure.  Let us explain the intuition behind this fact.  

When the entries in the signal decay rapidly, the algorithm must identify and remove the largest remaining entry from the residual before it can make further progress on the smaller entries.  Indeed, the large component in the residual contaminates each component of the signal proxy.  In this case, the algorithm may require an iteration or more to find each component in the signal.

On the other hand, when the $s$ entries of the signal are comparable, the algorithm can simultaneously locate many entries just by reducing the norm of the residual below the magnitude of the smallest entry.  Since the largest entry of the signal has magnitude at least $s^{-1/2}$ times the $\ell_2$ norm of the signal, the algorithm can find all $s$ components of the signal after about $\log s$ iterations. 

To quantify these intuitions, we want to collect the components of the signal into groups that are comparable with each other.  To that end, define the \term{component bands} of a signal $\vct{x}$ by the formulae
\begin{equation}\label{eq:bjs}
B_j \defby \left\{ i : 2^{-(j+1)} \enormsq{\vct{x}}
	< \abssq{ x_i } \leq 2^{-j} \enormsq{\vct{x}} \right\}
\qquad\text{for $j = 0, 1, 2, \dots$.}
\end{equation}
The \term{profile} of the signal is the number of bands that are nonempty:
$$
{\rm profile}(\vct{x}) \defby \# \{ j : B_j \neq \emptyset \}.
$$
In words, the profile counts how many orders of magnitude at which the signal has coefficients.  It is clear that the profile of an $s$-sparse signal is at most $s$.  See Figure~\ref{fig:profile} for images of stylized signals with different profiles.

\begin{figure}
\centering
\subfigure[Low profile]{
	\includegraphics[width=0.45\textwidth, height=2.5in]{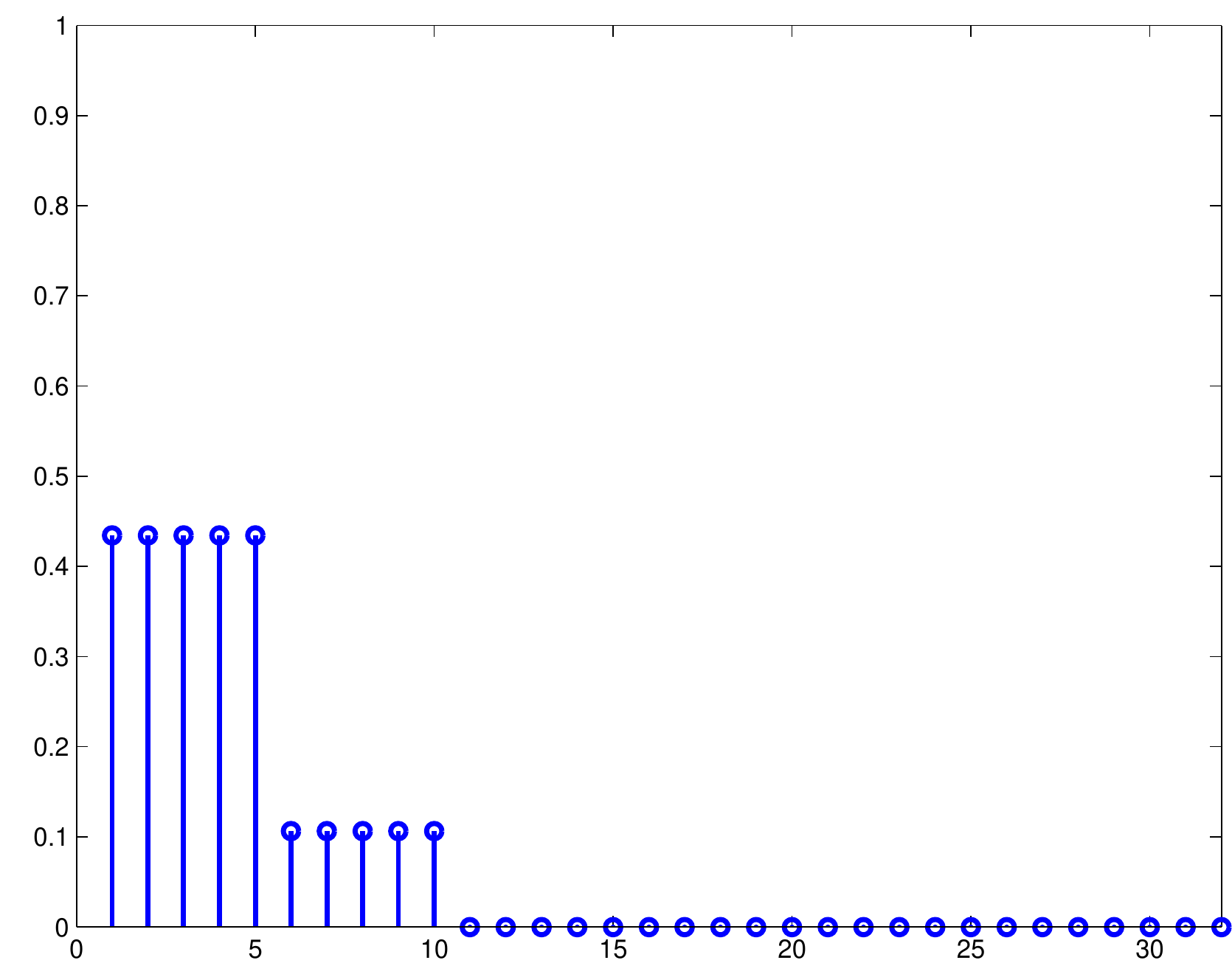}}
\hspace{1pc}
\subfigure[High profile]{
	\includegraphics[width=0.45\textwidth, height=2.5in]{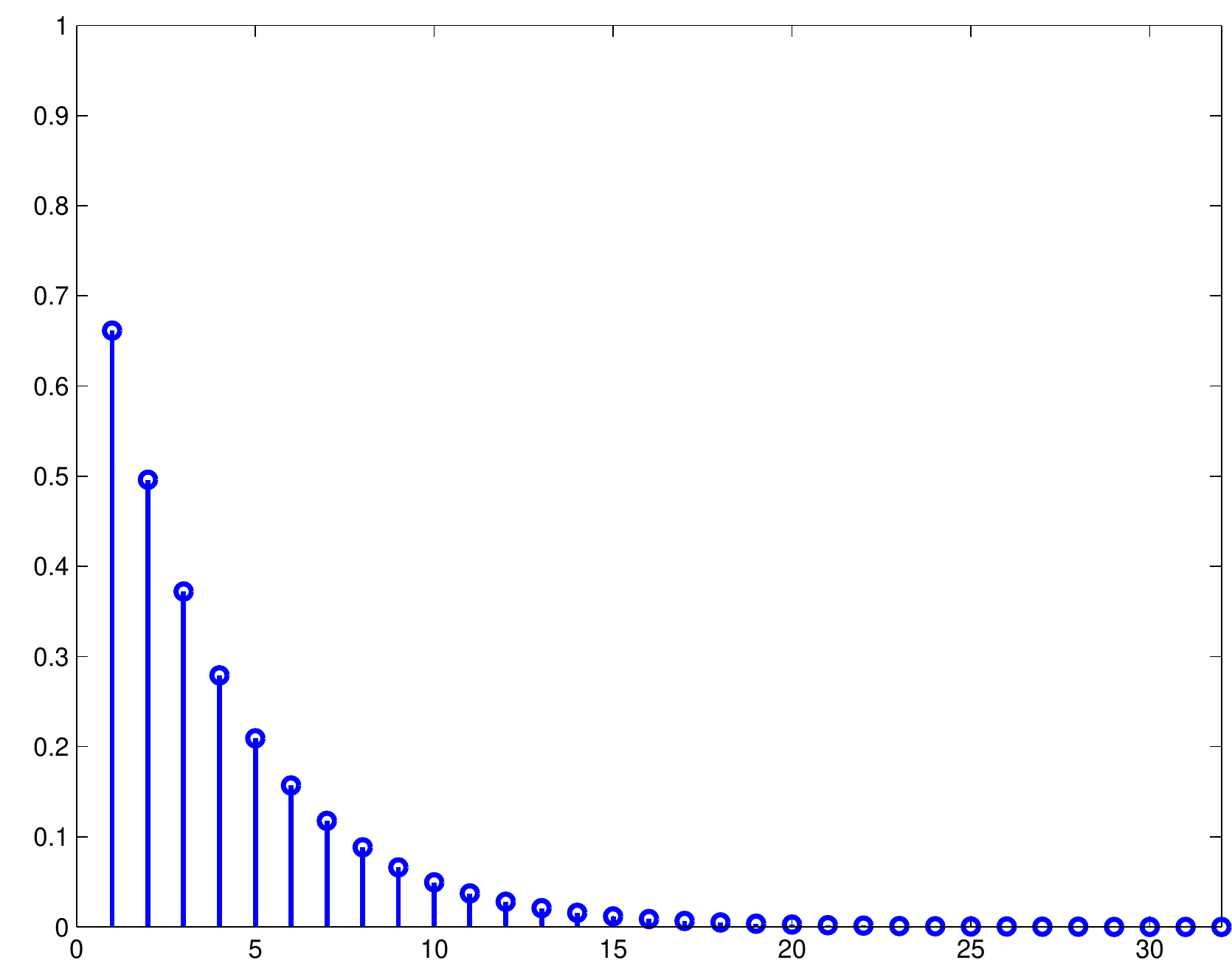}}
\caption{Illustration of two unit-norm signals with sharply different profiles.}  
	\label{fig:profile}
\end{figure}

First, we prove a result on the number of iterations needed to acquire an $s$-sparse signal.  At the end of the section, we extend this result to general signals, which yields Theorem~\ref{thm:cosamp-count}.


\begin{thm}[Iteration Count: Sparse Case] \label{thm:count-sparse}
Let $\vct{x}$ be an $s$-sparse signal, and define $p = {\rm profile}(\vct{x})$.  After at most
$$
p \log_{4/3}(1 + 4.6 \sqrt{s/p}) + 6
$$
iterations, \cosamp\ produces an approximation $\vct{a}$ that satisfies
$$
\enorm{ \vct{x} - \vct{a} } \leq 17 \enorm{ \vct{e} }.
$$
\end{thm}

For a fixed $s$, the bound on the number of iterations achieves its maximum value at $p = s$.  Since $\log_{4/3} 5.6 < 6$, the number of iterations never exceeds $6(s + 1)$.



\subsection{Additional Notation}

Let us instate some notation that will be valuable in the proof of the theorem.  We write $p = {\rm profile}(\vct{x})$.  For each $k = 0, 1, 2, \dots$, the signal $\vct{a}^k$ is the approximation after the $k$th iteration.  We abbreviate $S_k = \supp{ \vct{a}^k }$, and we define the residual signal $\vct{r}^k = \vct{x} - \vct{a}^k$.  The norm of the residual can be viewed as the approximation error.

For a nonnegative integer $j$, we may define an auxiliary signal
$$
\vct{y}^j \defby \vct{x}\restrict{\bigcup_{i \geq j} B_i}
$$
In other words, $\vct{y}^j$ is the part of $\vct{x}$ contained in the bands $B_j$, $B_{j+1}$, $B_{j+2}$, \dots.  For each $j \in J$, we have the estimate
\begin{equation} \label{eqn:signal-tail-bd}
\enormsq{ \vct{y}^j } \leq \sum\nolimits_{i \geq j} 2^{-i} \enormsq{\vct{x}} \cdot \abs{B_i}
\end{equation}
by definition of the bands.  These auxiliary signals play a key role in the analysis.

\subsection{Proof of Theorem~\ref{thm:count-sparse}}

The proof of the theorem involves a sequence of lemmas.  The first object is to establish an alternative that holds in each iteration.  One possibility is that the approximation error is small, which means that the algorithm is effectively finished.  Otherwise, the approximation error is dominated by the energy in the unidentified part of the signal, and the subsequent approximation error is a constant factor smaller.

\begin{lemma} \label{lem:iter-alter}
For each iteration $k = 0, 1, 2, \dots$, at least one of the following alternatives holds.  Either
\begin{equation} \label{eqn:resid-err-bd}
\smnorm{2}{ \vct{r}^k }
	\leq 70 \enorm{ \vct{e} }
\end{equation}
or else
\begin{align}
\smnorm{2}{ \vct{r}^k }
	&\leq 2.3 \smnorm{2}{ \vct{x}\restrict{S_k^c} }
	\qquad\text{and}
	 \label{eqn:missing-link} \\
\smnorm{2}{ \vct{r}^{k+1} }
	&\leq 0.75 \smnorm{2}{ \vct{r}^{k} }
	\label{eqn:error-reduct}.
\end{align}
\end{lemma}


\begin{proof}
Define $T_k$ as the merged support that occurs during iteration $k$.
The pruning step ensures that the support $S_{k}$ of the approximation at the end of the iteration is a subset of the merged support, so
$$
\smnorm{2}{ \vct{x}\restrict{T_{k}^c} } \leq
	\smnorm{2}{ \vct{x}\restrict{S_k^c} }
\qquad\text{for $k = 1, 2, 3, \dots$}.	
$$
At the end of the $k$th iteration, the pruned vector $\vct{b}_s$ becomes the next approximation $\vct{a}^{k}$, so the estimation and pruning results, Lemmas~\ref{lem:estimation} and~\ref{lem:pruning}, together imply that
\begin{align} 
\smnorm{2}{ \vct{r}^{k} }
	&\leq 2 \cdot ( 1.112 \smnorm{2}{ \vct{x}\restrict{T_{k}^c} }
		+ 1.06 \enorm{\vct{e}} ) \notag \\
	&\leq 2.224 \smnorm{2}{ \vct{x}\restrict{S_{k}^c} }
		+ 2.12 \enorm{\vct{e}} \label{eqn:est-bd}
\qquad\text{for $k = 1, 2, 3, \dots$}.
\end{align}
Note that the same relation holds trivially for iteration $k = 0$ because $\vct{r}^0 = \vct{x}$ and $S_0 = \emptyset$.

Suppose that there is an iteration $k \geq 0$ where
\begin{equation*} 
\smnorm{2}{ \vct{x}\restrict{S_{k}^c} } < 30 \enorm{\vct{e}}.
\end{equation*}
We can introduce this bound directly into the inequality \eqref{eqn:est-bd} to obtain the first conclusion \eqref{eqn:resid-err-bd}.

Suppose on the contrary that in iteration $k$ we have
\begin{equation*} 
\smnorm{2}{ \vct{x}\restrict{S_{k}^c} } \geq 30 \enorm{\vct{e}}.
\end{equation*}
Introducing this relation into the inequality \eqref{eqn:est-bd} leads quickly to the conclusion \eqref{eqn:missing-link}.  We also have the chain of relations
$$
\smnorm{2}{ \vct{r}^k }
	\geq \smnorm{2}{ \vct{r}^k\restrict{S_k^c} }
	= \smnorm{2}{ (\vct{x} - \vct{a}^k)\restrict{S_k^c} }
	= \enorm{ \vct{x}\restrict{S_k^c} }
	\geq 30 \enorm{ \vct{e}}.
$$
Therefore, the sparse iteration invariant, Theorem~\ref{thm:invar-sparse} ensures that \eqref{eqn:error-reduct} holds.
\end{proof}

The next lemma contains the critical part of the argument.  Under the second alternative in the previous lemma, we show that the algorithm completely identifies the support of the signal, and we bound the number of iterations required to do so.

\begin{lemma} \label{lem:iter-count}
Fix $K = \lfloor p \log_{4/3} (1 + 4.6 \sqrt{s/p}) \rfloor$.  Assume that \eqref{eqn:missing-link} and \eqref{eqn:error-reduct} are in force for each iteration $k = 0, 1, 2, \dots, K$.  Then $\supp{ \vct{a}^{K} } = \supp{\vct{x}}$.
\end{lemma}


\begin{proof}
First, we check that, once the norm of the residual is smaller than each element of a band, the components in that band persist in the support of each subsequent approximation.  Define $J$ to be the set of nonempty bands, and fix a band $j \in J$.  Suppose that, for some iteration $k$, the norm of the residual satisfies
\begin{equation} \label{eqn:resid-band-j}
\smnorm{2}{ \vct{r}^k } \leq 2^{-(j+1)/2} \enorm{ \vct{x} }.
\end{equation}
Then it must be the case that
$B_j \subset \supp{ \vct{a}^k }$.
If not, then some component $i \in B_j$ appears in the residual: $r^k_i = x_i$.  This supposition implies that
$$
\smnorm{2}{ \vct{r}^k } \geq \abs{x_i} > 2^{-(j+1)/2} \enorm{\vct{x}},
$$
an evident contradiction.  Since \eqref{eqn:error-reduct} guarantees that the norm of the residual declines in each iteration, \eqref{eqn:resid-band-j} ensures that the support of each subsequent approximation contains $B_j$.


Next, we bound the number of iterations required to find the next nonempty band $B_j$, given that we have already identified the bands $B_i$ where $i < j$.  Formally, assume that the support $S_k$ of the current approximation contains $B_i$ for each $i < j$.  In particular, the set of missing components $S_k^c \subset \supp{ \vct{y}^j }$.  It follows from relation \eqref{eqn:missing-link} that
$$
\smnorm{2}{\vct{r}^k} \leq 2.3 \enorm{ \vct{y}^j }.
$$
We can conclude that we have identified the band $B_j$ in iteration $k + \ell$ if
$$
\smnorm{2}{\vct{r}^{k+\ell}} \leq 2^{-(j+1)/2} \enorm{ \vct{x} }.
$$
According to \eqref{eqn:error-reduct}, we reduce the error by a factor of $\beta^{-1} = 0.75$ during each iteration.  Therefore, the number $\ell$ of iterations required to identify $B_j$ is at most
$$
\log_\beta \left\lceil \frac{ 2.3 \enorm{\vct{y}^j} }{ 2^{-(j+1)/2} \enorm{ \vct{x} } } \right\rceil
$$
We discover that the total number of iterations required to identify all the (nonempty) bands is at most
$$
k_{\star} \defby \sum\nolimits_{j\in J} \log_\beta \left\lceil 2.3 \cdot \frac{2^{(j+1)/2} \enorm{ \vct{y}^j }}{\enorm{ \vct{x} }} \right \rceil.
$$
For each iteration $k \geq \lfloor k_{\star} \rfloor$, it follows that $\supp{\vct{a}^k} = \supp{\vct{x}}$.

It remains to bound $k_{\star}$ in terms of the profile $p$ of the signal.  For convenience, we focus on a slightly different quantity.  First, observe that $p = \abs{J}$.  Using the geometric mean--arithmetic mean inequality, we discover that
\begin{align*}
\exp\left\{ \frac{1}{p} \sum\nolimits_{j \in J} \log \left\lceil
	2.3 \cdot \frac{2^{(j+1)/2} \enorm{ \vct{y}^j }}{\enorm{ \vct{x} }} \right\rceil
	\right\}
&\leq \exp\left\{ \frac{1}{p} \sum\nolimits_{j \in J} \log \left( 1 +
	2.3 \cdot \frac{2^{(j+1)/2} \enorm{ \vct{y}^j }}{\enorm{ \vct{x} }} \right)
	\right\} \\
&\leq \frac{1}{p} \sum\nolimits_{j \in J} \left( 1 +
	2.3 \cdot \frac{2^{(j+1)/2} \enorm{ \vct{y}^j }}{\enorm{ \vct{x} }} \right) \\
&= 1 + \frac{2.3}{p} \sum\nolimits_{j \in J}
	\left( \frac{ 2^{j+1} \enormsq{ \vct{y}^j }}{\enormsq{ \vct{x} }} \right)^{1/2}.
\end{align*}
To bound the remaining sum, we recall the relation \eqref{eqn:signal-tail-bd}.  Then we invoke Jensen's inequality and simplify the result.
\begin{align*}
\frac{1}{p} \sum_{j \in J}
	\left( \frac{ 2^{j+1} \enormsq{ \vct{y}^j }}{\enormsq{ \vct{x} }} \right)^{1/2}
&\leq \frac{1}{p} \sum\nolimits_{j \in J}
	\left( 2^{j+1} \sum\nolimits_{i \geq j} 2^{-i} \abs{B_i} \right)^{1/2}
\leq \left( \frac{1}{p} \sum\nolimits_{j\in J}
	2^{j+1} \sum_{i \geq j} 2^{-i} \abs{B_i} \right)^{1/2} \\
&\leq \left( \frac{1}{p} \sum\nolimits_{i \geq 0} \abs{B_i}
	\sum\nolimits_{j \leq i} 2^{j - i + 1} \right)^{1/2}
\leq \left( \frac{4}{p} \sum\nolimits_{i \geq 0} \abs{B_i} \right)^{1/2} \\
&= 2 \sqrt{s/p}.
\end{align*}
The final equality holds because the total number of elements in all the bands equals the signal sparsity $s$.  Combining these bounds, we reach
$$
\exp\left\{ \frac{1}{p} \sum\nolimits_{j \in J} \log \left\lceil
	2.3 \cdot \frac{2^{(j+1)/2} \enorm{ \vct{y}^j }}
		{\enorm{\vct{x}}} \right\rceil \right\}
\leq 1 + 4.6 \sqrt{s/p}.
$$
Take logarithms, multiply by $p$, and divide through by $\log \beta$.  We conclude that the required number of iterations $k_{\star}$ is bounded as
$$
k_{\star} \leq p \log_\beta( 1 + 4.6 \sqrt{s/p} ).
$$
This is the advertised conclusion.
\end{proof}

Finally, we check that the algorithm produces a small approximation error within a reasonable number of iterations.

\begin{proof}[Proof of Theorem~\ref{thm:count-sparse}]
Abbreviate $K = \lfloor p \log( 1 + 4.6 \sqrt{s/p} ) \rfloor$.
Suppose that \eqref{eqn:resid-err-bd} never holds during the first $K$ iterations of the algorithm.  Under this circumstance, Lemma~\ref{lem:iter-alter} implies that both \eqref{eqn:missing-link} and \eqref{eqn:error-reduct} hold during each of these $K$ iterations.  It follows from Lemma \ref{lem:iter-count} that the support $S_K$ of the $K$th approximation equals the support of $\vct{x}$.  Since $S_K$ is contained in the merged support $T_K$, we see that the vector $\vct{x}\restrict{T_K^c} = \vct{0}$.  Therefore, the estimation and pruning results, Lemmas~\ref{lem:estimation} and~\ref{lem:pruning}, show that
$$
\smnorm{2}{ \vct{r}^K }
	\leq 2 \cdot \left( 1.112 \enorm{ \vct{x}\restrict{T_K^c} }
		+ 1.06 \enorm{ \vct{e}} \right)
	= 2.12 \enorm{\vct{e}}.
$$
This estimate contradicts \eqref{eqn:resid-err-bd}. 

It follows that there is an iteration $k \leq K$ where  \eqref{eqn:resid-err-bd} is in force.  Repeated applications of the iteration invariant, Theorem~\ref{thm:invar-sparse}, allow us to conclude that
$$
\smnorm{2}{ \vct{r}^{K + 6} } < 17 \enorm{\vct{e}}.
$$
This point completes the argument.
\end{proof}



%

\subsection{Proof of Theorem~\ref{thm:cosamp-count}}

Finally, we extend the sparse iteration count result to the general case.

\begin{thm}[Iteration Count] \label{thm:count-detail}
Let $\vct{x}$ be an arbitrary signal, and define $p = {\rm profile}(\vct{x}_s)$.  After at most
$$
p \log_{4/3}(1 + 4.6 \sqrt{s/p}) + 6
$$
iterations, \cosamp\ produces an approximation $\vct{a}$ that satisfies
$$
\enorm{ \vct{x} - \vct{a} } \leq 20 \enorm{ \vct{e} }.
$$
\end{thm}

\begin{proof}
Let $\vct{x}$ be a general signal, and let $p = {\rm profile}(\vct{x}_s)$.  Lemma~\ref{lem:reduction} allows us to write the noisy vector of samples $\vct{u} = \Fee \vct{x}_s + \widetilde{\vct{e}}$.  The sparse iteration count result, Theorem~\ref{thm:count-sparse}, states that after at most
$$
p \log_{4/3}(1 + 4.6 \sqrt{s/p}) + 6
$$
iterations, the algorithm produces an approximation $\vct{a}$ that satisfies
$$
\enorm{ \vct{x}_s - \vct{a} } \leq 17 \enorm{ \widetilde{\vct{e}} }.
$$
Apply the lower triangle inequality to the left-hand side.  Then recall the estimate for the noise in Lemma~\ref{lem:reduction}, and simplify to reach
\begin{align*}
\enorm{ \vct{x} - \vct{a} }
	&\leq 17 \enorm{ \widetilde{\vct{e}} } + \enorm{ \vct{x} - \vct{x}_s } \\
	&\leq 18.9 \enorm{ \vct{x} - \vct{x}_s }
		+ \frac{17.9}{\sqrt{s}} \pnorm{1}{ \vct{x} - \vct{x}_s }
		+ 17 \enorm{ \vct{e} } \\
	&< 20 \nu,
\end{align*}
where $\nu$ is the unrecoverable energy.  
\end{proof}

\begin{proof}[Proof of Theorem~\ref{thm:cosamp-count}]
Invoke Theorem~\ref{thm:count-detail}.  Recall that the estimate for the number of iterations is maximized with $p = s$, which gives an upper bound of $6(s+1)$ iterations, independent of the signal.
\end{proof}


\bibliography{cosamp}

\begin{thebibliography}{10}

\bibitem{Bjo96:Numerical-Methods}
{\AA}.~Bj{\"o}rck.
\newblock {\em Numerical Methods for Least Squares Problems}.
\newblock SIAM, Philadelphia, 1996.

\bibitem{CRT06:Robust-Uncertainty}
E.~Cand{\`e}s, J.~Romberg, and T.~Tao.
\newblock Robust uncertainty principles: {E}xact signal reconstruction from
  highly incomplete {F}ourier information.
\newblock {\em IEEE Trans. Info. Theory}, 52(2):489--509, Feb. 2006.

\bibitem{CRT06:Stable}
E.~Cand\`es, J.~Romberg, and T.~Tao.
\newblock Stable signal recovery from incomplete and inaccurate measurements.
\newblock {\em Communications on Pure and Applied Mathematics},
  59(8):1207--1223, 2006.

\bibitem{Can08:Restricted-Isometry}
E.~J. Cand\`es.
\newblock The restricted isometry property and its implications for compressed
  sensing.
\newblock Submitted for publication, Feb. 2008.

\bibitem{CT05:Decoding-Linear}
E.~J. Cand{\`e}s and T.~Tao.
\newblock Decoding by linear programming.
\newblock {\em IEEE Trans. Info. Theory}, 51(12):4203--4215, Dec. 2005.

\bibitem{CT06:Near-Optimal}
E.~J. Cand{\`e}s and T.~Tao.
\newblock Near optimal signal recovery from random projections: {U}niversal
  encoding strategies?
\newblock {\em IEEE Trans. Info. Theory}, 52(12):5406--5425, Dec. 2006.

\bibitem{CDD06:Remarks}
A.~Cohen, W.~Dahmen, and R.~DeVore.
\newblock Compressed sensing and best $k$-term approximation.
\newblock Submitted for publication, 2006.

\bibitem{CLRS01:Intro-Algorithms}
T.~Cormen, C.~Lesierson, L.~Rivest, and C.~Stein.
\newblock {\em Introduction to Algorithms}.
\newblock MIT Press, Cambridge, MA, 2nd edition, 2001.

\bibitem{CM05:Combinatorial}
G.~Cormode and S.~Muthukrishnan.
\newblock Combinatorial algorithms for compressed sensing.
\newblock Technical report, DIMACS, 2005.

\bibitem{DM08:Subspace-Pursuit}
W.~Dai and O.~Milenkovic.
\newblock Subspace pursuit for compressive sensing: {C}losing the gap between
  performance and complexity.
\newblock Submitted for publication, Mar. 2008.

\bibitem{DDM04:Iterative-Thresholding}
I.~Daubechies, M.~Defrise, and C.~De Mol.
\newblock An iterative thresholding algorithm for linear inverse problems with
  a sparsity constraint.
\newblock {\em Comm. Pure Appl. Math.}, 57:1413--1457, 2004.

\bibitem{Don06:Compressed-Sensing}
D.~L. Donoho.
\newblock Compressed sensing.
\newblock {\em IEEE Trans. Info. Theory}, 52(4):1289--1306, Apr. 2006.

\bibitem{DTDS06:Sparse-Solution}
D.~L. Donoho, Y.~Tsaig, I.~Drori, and J.-L. Starck.
\newblock Sparse solution of underdetermined linear equations by stagewise
  {O}rthogonal {M}atching {P}ursuit ({StOMP}).
\newblock Submitted for publication, 2007.

\bibitem{FNW07:Gradient-Projection}
M.~A.~T. Figueiredo, R.~D. Nowak, and S.~J. Wright.
\newblock Gradient projection for sparse reconstruction: {A}pplication to
  compressed sensing and other inverse problems.
\newblock {\em IEEE J.~Selected Topics in Signal Processing: Special Issue on
  Convex Optimization Methods for Signal Processing}, 1(4):586--598, 2007.

\bibitem{GG84:On-widths}
A.~Garnaev and E.~Gluskin.
\newblock On widths of the {E}uclidean ball.
\newblock {\em Sov. Math. Dokl.}, 30:200--204, 1984.

\bibitem{GSTV07:Algorithmic}
A.~Gilbert, M.~Strauss, J.~Tropp, and R.~Vershynin.
\newblock Algorithmic linear dimension reduction in the $\ell_1$ norm for
  sparse vectors.
\newblock Submitted for publication, August 2006.

\bibitem{GSTV07:HHS}
A.~Gilbert, M.~Strauss, J.~Tropp, and R.~Vershynin.
\newblock One sketch for all: Fast algorithms for compressed sensing.
\newblock In {\em Proc. 39th ACM Symp. Theory of Computing}, San Diego, June
  2007.

\bibitem{GGIKMS02:Fast-Small-Space}
A.~C. Gilbert, S.~Guha, P.~Indyk, Y.~Kotidis, S.~Muthukrishnan, and M.~J.
  Strauss.
\newblock Fast, small-space algorithms for approximate histogram maintenance.
\newblock In {\em ACM Symposium on Theoretical Computer Science}, 2002.

\bibitem{GGIMS02:Near-Optimal-Sparse}
A.~C. Gilbert, S.~Guha, P.~Indyk, S.~Muthukrishnan, and M.~J. Strauss.
\newblock Near-optimal sparse {F}ourier representations via sampling.
\newblock In {\em Proc. of the 2002 ACM Symposium on Theory of Computing STOC},
  pages 152--161, 2002.

\bibitem{GMS03:Approximation-Functions}
A.~C. Gilbert, M.~Muthukrishnan, and M.~J. Strauss.
\newblock Approximation of functions over redundant dictionaries using
  coherence.
\newblock In {\em Proc. of the 14th Annual ACM-SIAM Symposium on Discrete
  Algorithms}, Jan. 2003.

\bibitem{GMS05:Improved}
A.~C. Gilbert, S.~Muthukrishnan, and M.~J. Strauss.
\newblock Improved time bounds for near-optimal sparse {F}ourier representation
  via sampling.
\newblock In {\em Proceedings of SPIE Wavelets XI}, San Diego, CA, 2005.

\bibitem{I07:sub-linear}
M.~Iwen.
\newblock A deterministic sub-linear time sparse {F}ourier algorithm via
  non-adaptive compressed sensing methods.
\newblock In {\em Proc. ACM-SIAM Symposium on Discrete Algorithms (SODA)},
  2008.

\bibitem{Kas77:The-widths}
B.~Kashin.
\newblock The widths of certain finite dimensional sets and classes of smooth
  functions.
\newblock {\em Izvestia}, 41:334--351, 1977.

\bibitem{KKL+06:Method-Large-Scale}
S.-J. Kim, K.~Koh, M.~Lustig, S.~Boyd, and D.~Gorinevsky.
\newblock A method for large-scale $\ell_1$-regularized least-squares problems
  with applications in signal processing and statistics.
\newblock Submitted for publication, 2007.

\bibitem{MZ93:Matching-Pursuits}
S.~Mallat and Z.~Zhang.
\newblock {M}atching pursuits with time-frequency dictionaries.
\newblock {\em IEEE Trans. Signal Process.}, 41(12):3397--3415, 1993.

\bibitem{Mil02:Subset-Selection}
A.~J. Miller.
\newblock {\em Subset Selection in Regression}.
\newblock Chapman and Hall, London, 2nd edition, 2002.

\bibitem{NV07:ROMP-Stable}
D.~Needell and R.~Vershynin.
\newblock Signal recovery from incomplete and inaccurate measurements via
  regularized orthogonal matching pursuit.
\newblock Submitted for publication, October 2007.

\bibitem{NV07:Uniform-Uncertainty}
D.~Needell and R.~Vershynin.
\newblock Uniform uncertainty principle and signal recovery via regularized
  orthogonal matching pursuit.
\newblock Submitted for publication, July 2007.

\bibitem{NN94:Interior-Point}
Y.~E. Nesterov and A.~S. Nemirovski.
\newblock {\em Interior Point Polynomial Algorithms in Convex Programming}.
\newblock SIAM, Philadelphia, 1994.

\bibitem{Paj08:Personal-Communication}
A.~Pajor.
\newblock Personal communication, Feb. 2008.

\bibitem{RG08:Sampling-Bounds}
G.~Reeves and M.~Gastpar.
\newblock Sampling bounds for sparse support recovery in the presence of noise.
\newblock Submitted for publication, Jan. 2008.

\bibitem{RV06:Sparse-Reconstruction}
M.~Rudelson and R.~Vershynin.
\newblock Sparse reconstruction by convex relaxation: {F}ourier and {G}aussian
  measurements.
\newblock In {\em Proc. 40th Annual Conference on Information Sciences and
  Systems}, Princeton, Mar. 2006.

\bibitem{Tem02:Nonlinear-Methods}
V.~Temlyakov.
\newblock Nonlinear methods of approximation.
\newblock {\em Foundations of Comput. Math.}, 3(1):33--107, 2003.

\bibitem{Tro04:Greed-Good}
J.~A. Tropp.
\newblock Greed is good: {A}lgorithmic results for sparse approximation.
\newblock {\em IEEE Trans. Info. Theory}, 50(10):2231--2242, Oct. 2004.

\bibitem{Tro07:Beyond-Nyquist-Talk}
J.~A. Tropp.
\newblock Beyond {N}yquist: Efficient sampling of sparse, bandlimited signals.
\newblock Presented at SampTA, June 2007.

\bibitem{TG07:Signal-Recovery}
J.~A. Tropp and A.~C. Gilbert.
\newblock Signal recovery from random measurements via orthogonal matching
  pursuit.
\newblock {\em IEEE Trans. Info. Theory}, 53(12):4655--4666, 2007.

\bibitem{TGMS03:Improved-Sparse}
J.~A. Tropp, A.~C. Gilbert, S.~Muthukrishnan, and M.~J. Strauss.
\newblock Improved sparse approximation over quasi-incoherent dictionaries.
\newblock In {\em Proc. 2003 IEEE International Conference on Image
  Processing}, Barcelona, 2003.

\end{thebibliography}

\end{document}